\documentclass{article}

\usepackage[a4paper, total={5.5in, 8in}]{geometry}
\usepackage[T1]{fontenc}
\usepackage[utf8]{inputenc}
\usepackage{comment}
\usepackage{amsfonts}
\usepackage{amsmath}
\usepackage{amsthm}
\usepackage{pgfplots}
\usepackage{tikz}
\usepackage{mathabx}
\usepackage{comment}

\usepackage{thmtools}
\usepackage{thm-restate}
\usepackage{hyperref}
\usepackage{cleveref}

\usepackage{graphicx,subcaption}
\usepackage{wrapfig}
\usepackage{tabstackengine}
\usepackage{mathtools}
\usepackage{bbm}

\newcommand{\sobonorm}[2]{\left|\left| #1 \right|\right|_{\dot H^{#2}}}

\newcommand{\PV}{{\text{P.V.}}}

\newcommand{\s}{\mathbf{s}}
\newcommand{\m}{\mathbf{m}}
\newcommand{\h}{\mathbf{h}}
\newcommand{\HH}{\mathcal{H}}

\newcommand{\oo}{\operatorname{o}}
\newcommand{\OO}{\operatorname{O}}

\newcommand{\iu}{{\mathrm{i}\mkern1mu}\hspace{0.02cm}}

\newcommand{\spinTE}[2]{
\dot{\mathbf{#1}}_j(t) = 2 \iu \sum_{k\neq j}^N \frac{\mathbf{#1}_j(t)\times \mathbf{#1}_k(t)}{\left( {#2}_j(t)-{#2}_k(t) \right)^2}
}
\newcommand{\poleTE}[2]{
\ddot{#2}_j(t) = -4 \sum_{k\neq j}^N \frac{\mathbf{#1}_j(t)\cdot \mathbf{#1}_k(t)}{\left( {#2}_j(t)-{#2}_k(t) \right)^3}
}

\newcommand{\vardbtilde}[1]{\tilde{\raisebox{0pt}[0.85\height]{$\tilde{#1}$}}}

\newcommand{\im}{\operatorname{Im}}
\newcommand{\re}{\operatorname{Re}}
\newcommand{\dd}{\operatorname{d}\hspace{-0.05cm}}

\usetikzlibrary{positioning} 
\usetikzlibrary{calc}
\usetikzlibrary{arrows,shapes,backgrounds}
\usepackage{biblatex}
\newtheorem{theorem}{Theorem}
\newtheorem{definition}{Definition}[section]

\newtheorem{lemma}[definition]{Lemma}
\newtheorem{corollary}[definition]{Corollary}
\newtheorem{proposition}[definition]{Proposition}

\title{Scattering of Rational Solutions to the Half-Wave Maps Equation}

\author{Gaspard Ohlmann}
\date{}

\begin{document}
	
	\maketitle	

        \section*{Abstract}

    This article studies the rational solutions of the Half-Wave Maps equation (HWM) in the non-singular spectrum case. We first provide characterizations to what we call \emph{scattering behavior}, and show that they imply scattering in Sobolev norm. We then provide a local condition implying \emph{scattering behavior}. Building on this, we show that any solution with non-singular spectrum scatters and provide an explicit formula for the function to which the solution is scattering. This allows us to show that the scattering map is the identity. Additionally, we create, for any given number of spins and any target non-singular spectrum, global solutions of (HWM) with a spectrum arbitrarily close to the target. Finally, we show, using a diagonal characterization of traveling waves, that if a wave scatters to a traveling wave, it is a scattering wave. \\

	\tableofcontents

    \section{Introduction}

    \subsection{The Half-Wave maps equation}

    In this article, we are interested in the half-wave maps equation that arises in multiple contexts, and can be seen as a continuum version of discrete completely integrable spin Calogero–Moser models \cite{zhou2015solitons} \cite{lenzmann2020derivation} \cite{berntson2022spin}. For a recent survey, we refer the reader to \cite{lenzmann2018short}. The Half-Wave maps equation reads
    \begin{equation}\tag{HWM-S}\label{HWMS}
        \partial_t \m(t,x) = \m(t,x) \times |\nabla| \m(t,x) = \m(t,x) \times \mathcal{H}(\partial_x m) (t,x),
    \end{equation}
    where $\m(t,\cdot)$ is a function from the real line to the sphere, i.e.
    \begin{equation*}
        m:(t,x) \in \mathbb{R}^2 \to \mathbb{S}^2,~\mathbb{S}^2 = \{ u \in \mathbb{C}^3, ~|u|=1 \},
    \end{equation*}
    and $\mathcal{H}$ is the Hilbert transform defined as 
    \begin{equation*}
        \mathcal{H}(f)(x) = \PV \frac{1}{\pi} \int_{y \in \mathbb{R}} \frac{f(y)}{x-y} \dd y.
    \end{equation*}
    In dimension one, this equation is energy-critical and preserves the norm
    
    \begin{equation*}
        ||f||_{\dot H^{1/2}(\mathbb{R})}^2 = \int_{\xi \in \mathbb{R}} |\hat f(\xi)|^2 |\xi| \dd\xi, \quad\text{where }
        \mathcal{F}(f)(\xi) = \hat f(\xi) = \frac{1}{\sqrt{2\pi}} \int_{x \in \mathbb{R}} e^{- \iu x \xi} f(x) \dd x.
    \end{equation*}

    We are interested in this article in the cases of rational functions, i.e. functions for which we can make the ansatz
    \begin{equation}\label{RATS}\tag{RAT-S}
            \m(t,x) = \m_0 + \sum_{j=1}^N \frac{\s_j(t)}{x-x_j(t)} + \sum_{j=1}^N \frac{\bar \s_j(t)}{x-\bar x_j(t)},~\im(x_j(t)) >0.
    \end{equation}
    
    It has been proven in \cite{gerard2018lax} that solutions that are rational for a given time $t_0\in \mathbb{R}$ stay rational for all times. As shown in \cite{berntson2020multi}, the study of this problem comes down to the study of a Calogero-Moser system of differential equations, studied for instance, in \cite{calo1}, \cite{calo2} or \cite{calo3}, where the dynamic of the spins and the poles is governed by 
    \begin{equation}\label{CMS}\tag{CM-S}
    \left\{
    \begin{aligned}
        &\spinTE{s}{x},~1 \leq j \leq N,\\
        &\poleTE{s}{x},~1 \leq j \leq N,
    \end{aligned}
    \right.
    \end{equation}
    and the dot product is defined without conjugation as $a \cdot b = \sum_{i=1}^3 a_i b_i$ for $a, b \in \mathbb{C}^3 $. If at $t_0$, $\m(t_0,x)$ belongs to the sphere for all $x$, then this property is preserved by the equation as $\partial_t \left(|\m|^2\right)=2 \langle \m \times |\nabla| \m,\m \rangle = 0$. Satisfying the additional constraint at any time $t_0$
    \begin{equation}\label{CONSS}\tag{CONS-S}
        \s_j(t_0)^2 =0,~ \s_j(t_0)\cdot \left(\m_0+\sum_{k=1}^N \frac{\s_k(t_0)}{x_j(t_0)-x_k(t_0)} + \sum_{k=1}^N \frac{\bar \s_k(t_0)}{x_j(t_0) - \bar x_k(t_0)} \right) =0,
    \end{equation}
    guarantees that the constraint, and this condition, is always satisfied for any time $t \in \mathbb{R}$. \\
    Only a handfull of solutions are known for this equation. In \cite{lenzmann2018energy}, Lenzmann and Schikorra provide an explicit characterization of the traveling waves, and in \cite{berntson2020multi}, the authors provide a class of solutions for which the expression is known for all times. In this article, we provide the existence of solutions whose spectrum are arbitrarily close to any given non-singular spectrum.
    
    In addition, well-posedness and behavior for large times are active subjects of research. In dimensions $d \geq 3$, global existence results are established for small data, see \cite{krieger2017small}, \cite{kiesenhofer2021small} or \cite{liu2021global}. In this article, we describe the scattering map for any solution with non-singular spectrum, and show that it is in fact the identity. We also provide an explicit formula for the function to which the solution is scattering.

    We now present an equivalent formulation found in \cite{gerard2018lax} - inspired for instance by the study of the classical Heisenberg model in \cite{takhtajan1977integration} - the $2\times 2$ Matrix formulation. With $\sigma_j \in \mathrm{SU}(2)$ the standard Pauli matrices, we can associate (bijectively) to a spin $\s \in \mathbb{C}^3$ its matrix representation $A\in \mathbb{C}^{2\times 2}$ through 
    \[
    A = \s \cdot \sigma = \sum_{j=1}^3 \s_j \sigma_j = \begin{pmatrix}
            \s_j[3] & \s_j[1]- i \s_j[2] \\
            \s_j[1] + i \s_j[2] & - \s_j[3]
        \end{pmatrix}.
    \]
    Then, the properties 
    \[
    \left\{
    \begin{aligned}
        &\left( \s\cdot \sigma \right) \left( \mathbf{t}\cdot \sigma \right) = (\s \cdot \mathbf{t}) I_N + i (\s \times \mathbf{t}) \cdot \sigma,~|\m|=1 \text{ if and only if } M^2 = I_N, \\
        &\left( \s \times \mathbf{t} \right) \cdot \sigma = -\frac{\iu}{2}\left[ \s\cdot \sigma, \mathbf{t}\cdot \sigma \right],
    \end{aligned}
    \right.
    \]
    allow us to rewrite \eqref{HWMS}, \eqref{RATS}, \eqref{CMS} and \eqref{CONSS} into their matrix variant, with $A_j = \s \cdot \sigma$ and $M = \m \cdot \sigma$. 
    
    \begin{equation}\label{HWMM}\tag{HWM-M}
        \partial_t \m(t,x) = -\frac{\iu}{2} \left[ \m(t,x),|\nabla| \m(t,x) \right],
    \end{equation}
    
    \begin{equation}\label{RATM}\tag{RAT-M}
        \m(t,x) = \m_0 + \sum_{j=1}^N \frac{A_j(t)}{x-x_j(t)} + \sum_{j=1}^N \frac{A_j^*(t)}{x-\bar x_j(t)},~\im (x_j(t)) >0,
    \end{equation}
    
    \begin{equation}\label{CMM}\tag{CM-M}
    \left\{
    \begin{aligned}
        &\partial_t A_j(t) = \sum_{k\neq j} \frac{\left[ A_j(t),A_k(t) \right]}{(x_j(t)-x_k(t))^2}\\
        &\ddot x_j(t) = -2 \sum_{k\neq j}^N \frac{Tr(A_j(t)A_k(t))}{(x_j(t)-x_k(t))^3}.
    \end{aligned}
    \right.
    \end{equation}
    
    \begin{equation}\label{CONSM}\tag{CONS-M}
    A_j^2 =0,~ A_j B_j + B_j A_j =0,\quad B_j = \m_0 + \sum_{k\neq j}^N \frac{A_k}{x_k-x_j} + \sum_{k=1}^N \frac{A_k^*}{x_j-\bar x_k}.
    \end{equation}

    In \cite{gerard2018lax}, with $a = a(t,x) \in \mathbb{R}^3$ a given function, $A(t,x)=a(t,x)\cdot \sigma$ its corresponding representation in the Lie algebra $\mathrm{SU}(2)$, and $\mu_A$ the operator
    \begin{equation*}
    \mu_A:
    \left\{
    \begin{aligned}
        &\phi \in \left(\mathbb{R}\to \mathbb{C}^2 \right) \mapsto \mu_A \phi \in \left(\mathbb{R}^2 \to \mathbb{C}^2 \right),\\
        &(\mu_A \phi) (t,x) = A(t,x) \phi(x),
    \end{aligned}
    \right.
    \end{equation*}
    Gérard and Lenzmann establish the Lax-Pair 
    \begin{equation*}
        \partial_t \mathcal{L} = [\mathcal{B},\mathcal{L}],
    \end{equation*}
    where $\mathcal{L} = [H,\mu_M]$ and $\mathcal{B}= -\frac{i}{2} (\mu_M |\nabla| + |\nabla|\mu_M) + \frac{i}{2} \mu_{|\nabla| M} $. This formulation provides that rational solutions stay rational and that $Tr(|\mathcal{L}|^k)$ are preserved quantities. It is noteworthy that the Half-wave maps equation can also be seen as the zero-dispersion limit of the spin Benjamin-Ono equation introduced in \cite{berntson2022spin} for which a Lax-pair has been found in \cite{gerard2023lax}. For a study of the scalar zero-dispersion limit of the Benjamin-Ono equation, we refer the reader to \cite{gerard2024zero}.

    Finally, we introduce the notations and results from \cite{ohlmann2024halfwavemapsexplicitformulas} that we will use. In particular, the author introduces the half-spin formulation to establish an explicit formula for \eqref{HWMM}. Namely, for a rational solution $\m$ of \eqref{HWMM} given by \eqref{RATM}, there exists $N$ diagonal matrices $E_i(t)\in \mathbb{C}^{2\times 2}$ and $N$ diagonal matrices $F_i(t) \in \mathbb{C}^{2\times 2}$ such that for all $i$,
    \[
    E_i(t) H F_i(t) = A_j(t),\quad \text{ where } H = \begin{pmatrix}
        1 & 1 \\
        1 & 1 
    \end{pmatrix}.
    \]
    Morever, for such $E_i$ and $F_i$, with $\mathcal{E}(t) = \operatorname{Diag}(E_1,\dots,E_N)$ and $\mathcal{F}(t) = \operatorname{Diag}(F_1,\dots,F_N)$,
    $\mathcal{E}_0= \mathcal{E}(0)$ and $\mathcal{F}_0 = \mathcal{F}(0)$, the solution is given by 
    \[
    \Pi_- \left( \m(t,x) - \m_0 \right) = \sum_{j=1}^N \frac{A_j(t)}{x-x_j(t)} = - T^T \mathcal{E}_0 \mathcal{H} [X(0) + t L(0) - xI_N]^{-1} \mathcal{F}_0 T,
    \]
    where $\mathcal{H}=\operatorname{Diag}(H,\dots,H)\in \mathbb{C}^{2N \times 2N}$, $T$ is a column vector given by $N$ times the $2\times 2$ identity matrix, and for a matrix $P$, $[P]\in \mathbb{C}^{2N\times N}$ is the doubled matrix of $P$, composed of $N\times N$ blocs, where the block at position $(i,j)$ is given by $P_{ij} I_2$. Finally, with 
    \[
    E_i = \begin{pmatrix}
        e_i^1 & 0 \\
        0 & e_i^2
    \end{pmatrix},\quad F_i = \begin{pmatrix}
        \xi_i^1 & 0 \\
        0 & \xi_i^2
    \end{pmatrix},
    \]
    the matrix $L(0)$ is given by 
     \[
    L_{jk} = \delta_{j,k} \dot x_j(0) + (1-\delta_{jk}) \frac{\xi_j(0) \cdot e_k(0)}{x_j(0)-x_k(0)},\quad a \cdot b = \sum_{i} a_i b_i.
    \]
    This expression has also been proven correct in \cite{gerard2024global}. Note that $L$ is not hermitian. This expression of $L$ coincides with the one provided by Matsuno in \cite{matsuno2022integrability}. In this work, Matsuno provides a Lax-pair for this equation. By defining the following two matrices
\begin{equation}\label{defLBMatsuno}
    L_{i,j}(t) = \left\{
    \begin{aligned}
    & \dot x_i(t),~ i=j,\\
    &\varepsilon_{i,j} \frac{\sqrt{ - 2 s_i(t)\cdot s_j(t)}}{(x_i(t) - x_j(t))},~ i\neq j,
    \end{aligned}
    \right.
    \quad
    B_{i,j}(t) = \left\{
    \begin{aligned}
    & 0,~ i=j,\\
    &\varepsilon_{i,j} \frac{\sqrt{ - 2 s_i(t)\cdot s_j(t)}}{(x_i(t) - x_j(t))^2},~ i\neq j,
    \end{aligned}
    \right.
\end{equation}
he finds that the matrix $L(t)$ satisfies 

\begin{equation}\label{LaxMatsuno}
    \partial_t L(t) = [B(t),L(t)] = B(t) L(t) - L(t) B(t).
\end{equation}
With 

\begin{equation}\label{defS}
    S_{i,j}(t) = \left\{
    \begin{aligned}
        &0,~i=j,\\
        &\sqrt{-2 \s_i(t) \cdot \s_j(t)},~i\neq j,
    \end{aligned}
    \right.
\end{equation}
 $X(t)$ being a diagonal matrix such that $X_{j,j}(t) = x_j(t)$ and $U(t)$ the solution of $U(0)=I_N$ and $\dot U(t) = B(t) U(t)$, he establishes
\begin{equation}\label{propMatsuno}
    \left\{
        \begin{aligned}
            &L(t) = U(t) L(0) U(t)^{-1},\\
            &S(t) = U(t) S(0) U(t)^{-1},\\
            &X(t) = U(t) \left( X(0) + t L(0) \right) U(t)^{-1}.
        \end{aligned}
    \right.
\end{equation}
    and obtains in particular that 
    $X(t) \sim X(0)+tL(0)$. We will hence say that the solution $\m$ is \textbf{non-singular} when $v_1,\dots,v_N$, the eigenvalues of $L(0)$, are all distinct. This is equivalent to the eigenvalues of $L(t)$ being distinct for all times.

	\subsection{Results}

    We first prove in Section \ref{sec:local} the following theorem, providing a local sufficient condition for scattering. 

    \begin{theorem}\label{thm:localCondition}
        Let $\m$ be a rational solution of \eqref{HWMS} and define $x_j$, $\s_j$ as 
        \[
        \m(t,x) = \m_0 + \sum_{j=1}^N \frac{\s_j(t)}{x-x_j(t)} + \sum_{j=1}^N \frac{\bar \s_j(t)}{x-\bar x_j(t)}.
        \]
        Further define $S$, $\nu$, $D$ as
        \begin{equation}
        \left\{
        \begin{aligned}
            &S(t)x=\max_{j} |\s_j(t)|,\\
            &\nu(t)=\min_{j\neq k} | \re(\dot x_j(t)) -  \re(\dot x_k(t))|,\\
            &D(t)=\min_{j\neq k} |\re(x_j(t))- \re(x_k(t))|.
        \end{aligned}
        \right.
        \end{equation}
	    Then, if there exists $t_0$ such that 
        \begin{equation}
        \alpha(t_0) = D(t_0) - \frac{2^4 N S(t_0)}{\nu(t_0)} > 0,
        \end{equation}
        then $\m$ scatters to some function 
        \begin{equation}
        g(t,x) = \m_0 + \sum_{j=1}^N \frac{\textbf{a}_j}{x-v_j t - b_j} + \sum_{j=1}^N \frac{\bar{\textbf{a}}_j}{x-v_j t - \bar{b}_j}.
        \end{equation}
        Moreover, $x_j$ and $\s_j$ satisfy 
        \begin{equation}
        \left\{
        \begin{aligned}
            &x_j(t) - v_j t  = b_j + \oo(1),\\
            &\s_j(t) = \textbf{a}_j + \oo(1).
        \end{aligned}
        \right.
        \end{equation}
        
    \end{theorem}

    Section \ref{sec:creation} is devoted to the creation, for any number of solitons $N$ and any target spectrum $(\omega_1,\dots,\omega_N)$, a solution that almost satisfies the requirement. More precisely, we prove the following theorem.

    \begin{theorem}\label{thm:creationForN}
        Let $N\in \mathbb{N}$, $\varepsilon>0$ and $(\omega_1,\dots,\omega_N)$ satisfying $\omega_j\neq \omega_k$ for $j\neq k$. There exists a global for positive times solution $\m$ to \eqref{HWMS} such that the spectrum $(v_1,\dots,v_N)$ of $\m$ satisfies
        \[
        |v_i-\omega_i| \leq \varepsilon.
        \]
        Furthermore, with 
        \[
        \m(t,x) = \m_0 + \sum_{j=1}^N \frac{\s_j(t)}{x-x_j(t)} + \sum_{j=1}^N \frac{\bar \s_j(t)}{x-\bar x_j(t)},
        \]
        we have 
        \[
        |\dot x_j(t) - \omega_j| \leq \varepsilon.
        \]
    \end{theorem}
    
    In Section \ref{sec:sys}, we use the local condition to show that any function with a non-singular spectrum necessarily scatters. More precisely, we show

    \begin{theorem}\label{thm:systematicScattering}
        Let $\m$ be a rational function of \eqref{HWMS} with spins $\s_j$ and poles $x_j$, of non-singular spectrum. Then, there exists $t_0$ such that 
        \[
        \alpha(t_0) > 0.
        \]
        Subsequently, there exist $a_j\in \mathbb{C}$, $\textbf{b}_j\in \mathbb{C}^3$ with $\im(a_j)\neq 0$ such that for $s \geq 0$,
        \[
        ||\\m(t,x)-g(x,t)||_{H^s(\mathbb{R})} \xrightarrow[t \rightarrow \infty]{} 0,
        \]
        where
        \[
        \left\{
        \begin{aligned}
            &g(x,t) = \m_0 + \sum_{j=1}^N \frac{\textbf{b}_j}{x-a_j - v_j t} + \sum_{j=1}^N \frac{\bar {\textbf{b}}_j}{x-\bar a_j- v_j t},\\
            &x_j(t)-v_j t \xrightarrow[t\rightarrow \infty]{} a_j,\\
            &\s_j(t) \xrightarrow[t\rightarrow \infty]{} \textbf{b}_j.
        \end{aligned}
        \right.
        \]
    \end{theorem}

    Moreover, we provide an explicit formula for $a_j$ and $\textbf{b}_j$ depending on the configuration of $\m$ at a given time $t$.

    Additionally, we also show in Section \ref{sec:sys} that the scattering map is trivial for functions with non-singular spectrum. More precisely, we show

    \begin{theorem}\label{thm:scatteringTrivial}
        Let $\m$ be a rational solution of \eqref{HWMS} with non-singular spectrum. There exists a rational function $g$
        \[
        g(t,x) =\m_0 + \sum_{j=1}^N \frac{\textbf{b}_j}{x-a_j - v_j t} + \sum_{j=1}^N \frac{\bar {\textbf{b}}_j}{x-\bar a_j- v_j t},
        \]
        such that $\m$ scatters to $g$ in Sobolev norm both as $t\rightarrow +\infty$ and $t\rightarrow -\infty$ i.e. 
        \[
        \begin{aligned}
            ||\m(t,x)-g(t,x)||_{H^s(\mathbb{R})} \xrightarrow[t \rightarrow \pm\infty]{} 0,~\forall s\geq 0.
        \end{aligned}
        \]
        
    \end{theorem}

    Finally, in Section \ref{sec:matTrav}, we consider the transfer from infinite times to finite times of travelingness. We answer the question: \emph{If $f$ scatters to $g$ at $+\infty$ and $g$ is a traveling wave, is $f$ necessarily a traveling wave?} by showing the following theorem, providing to this end a diagonal characterization of traveling waves.

    \begin{restatable}[]{theorem}{trav}
    	
\label{thm:scatterTrav}
    If $\m$ is a solution of \eqref{HWMS} and scatters to a traveling wave, then $\m$ is a traveling wave for finite times. More precisely, we consider a rational solution $\m$ of \eqref{HWMM} and define $\m_0$, $A_j$,$x_j$ and $V$ as

    \[
    \m(t,x) = \m_0 + V(t,x) + V^*(t,x),\quad V(t,x) = \sum_{j=1}^N \frac{A_j(t)}{x-x_j(t)}
    \]
    We consider a traveling wave of the form 
    \[
    G(t,x) = \m_0 + Q_v(x-tv) + Q^*_v(x-tv),\quad Q_v(z)=\sum_{j=1}^N \frac{C_j}{z-y_j}.
    \]
    where $y_j\neq y_k$ for $j \neq k$. We assume 

    \begin{equation}
        \left|\left| F(t,x)- G(t,x) \right|\right|_{\dot H^{1/2}(\mathbb{R})}\rightarrow 0.
    \end{equation}
    Then at any time, $\m$ was already traveling, meaning that with $L$, $B$ and $S$ defined as in \eqref{defLBMatsuno} and \eqref{defS}, respectively, 
    
    \begin{equation}
        L(t) = v I_N,~ B(t) = 0,~ \dot X(t) = v I_N, S(t)=0.
    \end{equation}
    \end{restatable}
    
	\section{Scattering behavior and Local condition}\label{sec:local}
	
	The goal of this section is to prove Theorem \ref{thm:localCondition}. To do so, we proceed by steps and start by introducing characterizations of \emph{scattering behavior} and show that they imply scattering. The main part of this section is then to prove Proposition \ref{scattering} which a stronger version of Theorem \ref{thm:localCondition}.
    
	\subsection{Behavior assuming a scattering situation}
	
	In this subsection, we introduce what we call a \emph{scattering situation}, and provide characterizations. We show that those characterizations are equivalent.
	Throughout this subsection, $\m$ will be a rational solution of \eqref{HWMS} and $x_j(t)$ and $s_j(t)x$ are defined as
        \[
        \m(t,x) = \m_0 + \sum_{j=1}^N \frac{\s_j(t)}{x-x_j(t)} + \sum_{j=1}^N \frac{\bar \s_j(t)}{x-\bar x_j(t)}.
        \]

	\begin{definition}
		We will say that $\m$ is in a scattering situation, denoted condition $(C_3)$ if and only if nothing will ever happen after a certain time, that we describe as

        \begin{equation}\label{C3}
            (C_3):=~ \left\{
            \begin{aligned}
                &(i)~x_j(t) = v_j t + x_j^\infty + \oo(1)\\
                &(ii)~x_j^\infty = R_j + \iu I_j,~I_j\in\mathbb{R}^{+}_*,~R_j\in \mathbb{R},\\
                &(iii)~s_j(t)=s_j^\infty +\oo(1),\\
                &(iv)~\dot x_j(t) = v_j + \oo(1).
            \end{aligned}
            \right.
            \quad\text{as }t\rightarrow \infty
        \end{equation}

    \end{definition}

        Note that $(ii)$ contains that the limit of the imaginary part of the poles stays above $0$, which ensures non-turbulence and boundedness of all Sobolev norms, as we will show later.

        We now present a theorem introducing two weaker conditions, denoted $(C_1)$ and $(C_2)$, that describes scattering behaviors, that proves that they are in fact, equivalent.

        \begin{lemma}
            We consider $\m$ a solution of \eqref{HWMS} and keep the same notations. We define the two conditions, $(C_1)$ and $(C_2)$, as

            \begin{equation}
                (C_1):=~ \left\{
                \begin{aligned}
                    &\exists \eta>0,~|x_j(t)-x_k(t)|\geq \eta t,\text{ for }t>t_0,~j\neq k\\
                    &\exists \mathcal{S},~|s_j(t)| \leq \mathcal{S},\text{ for }t>t_0,
                \end{aligned}
                \right.
                ,\quad \text{ for some }t_0\in \mathbb{R},
            \end{equation}

            and 

            \begin{equation}
                (C_2):=~\left\{
                \begin{aligned}
                    &x_j(t) = v_j t + \OO(1),~t\rightarrow \infty \\
                    &\exists \mathcal{S},~|s_j(t)| \leq \mathcal{S},\text{ for }t>t_0,
                \end{aligned}
                \right.
                ,\quad \text{ for some }t_0>0.
            \end{equation}  
            Note that condition $(C_1)$ is not minimal, as studied in \cite{ohlmann2024studywellposedness1denergycritical}. Then, We have 
            \begin{equation}
                (C_1) \Leftrightarrow (C_2) \Leftrightarrow (C_3).
            \end{equation}            
        \end{lemma}

        \begin{proof}
            The directions $(C_3) \Rightarrow (C_2) \Rightarrow (C_1)$ are immediate. Hence, it is enough to show $(C_1)\Rightarrow (C_3)$. 

            We consider $\m$ a solution of \eqref{HWMS}, $\eta>0$ and $\mathcal{S}>0$ such that $(C_1)$ is satisfied, i.e., for some $t_0$ and $t>t_0$,
            \begin{equation}
                \left\{
                \begin{aligned}
                    &|x_j(t)-x_k(t)| \geq \eta t,\\
                    &|s_j(t) \leq \mathcal{S}.
                \end{aligned}
                \right.
            \end{equation}  
            The time evolution equation of the poles reads

            \begin{equation}
                \poleTE{s}{x},
            \end{equation}
            which implies in particular for $t_1,t_2>t_0$,

            \begin{multline}
                |\dot x_j(t_2) - \dot x_j(t_1)| = \left| -4 \int_{t=t_1}^{t_2} \sum_{k\neq j}^N \frac{s_j(t)\cdot s_k(t)}{(x_j(t)-x_k(t))^3} \right|
                \leq 4 (N-1) \mathcal{S}^2 \int_{t=t_1}^{t_2} \frac{1}{\eta^3 |t-t_1|^3}  \dd t \\
                \leq 4 (N-1) \mathcal{S}^2 \frac{1}{\eta^3} \left( \frac{1}{t_1^2} - \frac{1}{t_2^2} \right) \leq 4 \mathcal{S}^2 \sum_{k\neq j}^N \frac{1}{\eta^3} \frac{1}{t_1^2}.
            \end{multline}

            In particular, $\dot x_j(t)$ is a Cauchy sequence, and converges to some $\tilde v_j=v_j$, which proves $(iv)$. The same equation gives
            \begin{equation}
                |\dot x_j(t) - v_j| \leq 4 (N-1) \mathcal{S}^2 \frac{1}{\eta^3} \frac{1}{t^2},
            \end{equation}
            so the function $t\mapsto x_j(t)-v_j t$ is an $L^1$ function, and the poles satisfy
            \begin{equation}
                x_j(t)-v_j t=x_j^\infty + \oo(1),
            \end{equation}
            which proves $(i)$. We now consider the spins. The time evolution equation reads
            \begin{equation}
                \spinTE{s}{x},
            \end{equation}
            which implies that for $t_1,t_2>t_0$,
            \begin{multline}
                |s_j(t_2)-s_j(t_1)| = \left| 2 \iu \int_{t=t_1}^{t_2} \sum_{k\neq j}^N \frac{s_j(t)\times s_k(t)}{(x_j(t)-x_k(t))^2} \right| \\
                \leq 2 (N-1) \mathcal{S}^2 \int_{t=t_1}^{t_2} \frac{1}{\eta^2 |t-t_1|^2} \dd t
                \leq 2 (N-1) \mathcal{S}^2 \frac{1}{\eta^2} \frac{1}{t_1} \leq 2 (N-1) \mathcal{S}^2 \frac{1}{\eta^2} \frac{1}{t_0}.
            \end{multline}
            Hence, the spins are a Cauchy sequence, and converge to $s_j^\infty$, which concludes the proof of $(iii)$. This also provides boundedness of the spins. The only thing left to show is that the imaginary parts of the poles stay above $0$. We will use one of the constraints from \eqref{CONSS} satisfied by $\m$:
			
			\begin{equation}
				\s_j(t) \cdot \left( \iu \m_0 - \sum_{k \neq j }^N \frac{\s_k(t)}{x_j(t) - x_k(t)} + \sum_{k=1}^N \frac{\bar \s_k(t) }{ x_j(t) - \bar x_k(t)} \right) = 0.
			\end{equation}
			Using
			\begin{equation}
				\left\{ 
				\begin{aligned}
					&\s_j(t) = \s_j(\infty) + \oo(1), \\
					&\s_k(t) = \s_k(\infty) + \oo(1), \\
					&x_j(t) - x_k(t) = t (v_j - v_k) + \oo(t),
				\end{aligned}
				\right.
			\end{equation}
			yields
			
			\begin{multline}
				\s_j(t) \cdot \left( \iu \m_0 -\sum_{k \neq j}^N \frac{\s_k(\infty)+\oo(1)}{t (v_j-v_k) + \oo(t)} + \sum_{k \neq j}^N \frac{\bar \s_k(\infty)+\oo(1)}{t (v_j-v_k) + \oo(t)}  \right) \\ - 
				\frac{\iu |\s_j(\infty)|^2  + \oo(1)}{ 2 \im(x_j(t))}=0.
			\end{multline}
			Now, we have
			
			\begin{equation}
				\left| \s_j(t) \cdot \left( \iu \m_0 -\sum_{k \neq j}^N \frac{\s_k(\infty)+\oo(1)}{t (v_j-v_k) + \oo(t)} + \sum_{k \neq j}^N \frac{\bar \s_k(\infty)+\oo(1)}{t (v_j-v_k) + \oo(t)}  \right) \right| \leq 2 |\s_j(\infty)|.
			\end{equation}
			We hence obtain $\frac{|\s_j(\infty)|^2}{2 \im(x_j(t))} \leq 4 |\s_j(\infty)|+\oo(1)$, so $
				|\im(x_j(t))| \geq \frac{|\s_j(\infty)|}{8}+\oo(1)$.	 Since the sum of the speeds is constant and real, i.e. $\sum_{j=1}^N \dot x_j(t) = \sum_{j=1}^N v_j \in \mathbb{R}$, the sum of the imaginary parts of the poles is constant as well. Taking into account the non-negativity of the imaginary parts, $\im(x_j(t)) \geq 0$, we have that for any $j$, $|\im(x_j(t))|\leq \sum_{j=1}^N \im(x(0))$. Hence,
			
			\begin{equation}
				\frac{|\s_j(\infty)|}{8}	\leq |\im(x_j(t))| \leq \sum_{j=1}^N \im(x(0))
			\end{equation}
			Using the constraint once more gives 
			
			\begin{equation}
				\frac{|\s_j(\infty)|^2}{2 \im(x_j(t))} = \s_j(\infty) \cdot \m_0 + \oo(1),
			\end{equation}
			First, we obtain that $\s_j(\infty) \cdot \m_0 \neq 0$, and also
			
			\begin{equation}
				\im(x_j(t)) \rightarrow \im(x_j(\infty)) = \frac{|\s_j(\infty)|^2}{\s_j(\infty) \cdot \m_0}.
			\end{equation}
			For this, we have used several times that for any $j$, $s_j(\infty) \neq 0$. We will show it now, and we continue to assume that the spins are bounded. Using the time evolution equation of the spins and the boundedness of the spins, we obtain for $t>t_0$,
			
			\begin{equation}
				\frac{\partial }{\partial t} |\s_j(t)| \leq C \mathcal{S} |\s_j(t)|,
			\end{equation}
			This differential equation is not compatible with $s_j(\infty)=0$, which gives $|s_j(\infty)|>0$ and $|\im(x_j(\infty))| >0$.
        
        \end{proof}

    \subsection{Scattering behavior implies scattering}

    In this subsection, we show that the scattering behavior implies that the solution is scattering in Sobolev norm to some function $g$ as $t\rightarrow \infty$. We start by computing Sobolev norms of the involved functions; we refer the reader to the appendix for the proofs.
	
	\begin{lemma}\label{scalar}
		Let $\frac{\s_i}{x-x_i}$ and $\frac{\s_j}{x-x_j}$ be two solitons. Then
		
		\begin{equation}
			\langle \frac{\s_i}{x-x_i} , \frac{\s_j}{x-x_j} \rangle = \int_{-\infty}^\infty \frac{\s_i}{x-x_i} \frac{\bar \s_j}{x-\bar x_j} dx = \frac{\s_i \bar \s_j}{x_i - \bar x_j} 2 i \pi.
		\end{equation}
		
	\end{lemma}

    \begin{proof}
        See Appendix \ref{proof1}.
    \end{proof}
	
	\begin{corollary}
		The $L^2$ norm of the solution $\m$ satisfies
		
		\begin{equation}\label{L2}
			||\m(t)||_{L^2}^2 \xrightarrow{t\rightarrow \infty} \sum_{i=1}^N \frac{|s_i(\infty)|^2 \pi}{\im(x_i)(\infty)} < \infty.
		\end{equation}
		
	\end{corollary}
	
	\begin{lemma}\label{technical}
		For $k \in \mathbb{N}^*$, 
		\begin{equation}\label{formulederiv}
			\left( \frac{\partial}{\partial x} \right)^k \left( \frac{\s_i}{x-x_i} \right) = (-1)^k k! \frac{\s_i}{(x-x_i)^{k+1}},
		\end{equation}
	
		\begin{equation}\label{formule2}
			\langle \frac{\s_i}{(x-x_i)^{k+1}} , \frac{\s_i}{(x-x_i)^{k+1}} \rangle = \frac{|s_i|^2}{\im(x_i)^{2k+1}} \frac{\sqrt{\pi} \Gamma(k+\frac{1}{2})}{\Gamma(k+1)},
		\end{equation}
	
	\begin{equation}\label{formule3}
		\langle \frac{\s_i}{(x-x_i)^{k+1}}, \frac{\s_j}{(x-x_j)^{k+1}} \rangle \xrightarrow{t\rightarrow \infty} 0.
	\end{equation}
	
	Subsequently, we have
	
	\begin{equation}\label{Hk}
		\left|\left| \left( \frac{\partial}{\partial x} \right)^k  \m(t)\right|\right|_{L^2}^2 \xrightarrow{t\rightarrow \infty} \sum_{j=1}^N k! \frac{|s_i(\infty)|^2\sqrt{\pi}\Gamma(k+\frac{1}{2})}{\im(x_i(\infty))^{2k+1}} < \infty.
	\end{equation}
	
	\eqref{L2} and \eqref{Hk} give that all Sobolev norms of the solution are bounded.
	
	\end{lemma}

    \begin{proof}
        See Appendix \ref{proof2}.
    \end{proof}

    We now consider a solution $\m$ of \eqref{HWMS} that possesses a scattering behavior, so it satisfies $(C_1)$, $(C_2)$ and $(C_3)$, and show that it has to scatter in Sobolev norm for any Sobolev norm. 

	\begin{proposition}\label{sec6:sysScatt}
		Let $\m$ be a rational solution and define $x_j$ and $\s_j$ as
		\[
		\m(t,x) = \m_0 + \sum_{j=1}^N \frac{\s_j(t)}{x-x_j(t)} + \sum_{j=1}^N \frac{\bar{\s}_j(t)}{x-\bar x_j(t)}.
		\]
		Further define $a_j$ the limit of $x_j(t)-v_j t$, $\textbf{b}_j$ the limit of $\s_j(t)$ and the function $F(t,x)$ as
				\begin{equation}
			F(x,t) = \sum_{j=1}^N \frac{b_j}{x - v_j t - a_j}.
		\end{equation}
		Then, $\m$ scatters in Sobolev norm. Precisely, for any Sobolev norm $\mathcal{N}$ of the form $\dot H^s$ or $H^s$ for $s\geq 0$, we have
		\begin{equation}
			||\m(t) - F(t)||_{\mathcal{N}} \xrightarrow{t\rightarrow \infty} 0.
		\end{equation}
	\end{proposition}

	To show this, we will show the following lemma
	
	\begin{lemma}
	Using the same notation we have $||\m(t) - F(t)||_{L^2} \xrightarrow{t\rightarrow \infty} 0$, and for any integer $k\geq 1$,
	
	\begin{equation}\label{sobo}
		\left|\left|  \left( \frac{\partial}{\partial x} \right)^k \left( \m(t) - F(t) \right) \right|\right|_{L^2} \rightarrow 0.
	\end{equation}
	
	\end{lemma}

	\begin{proof}
		We define for $i \in \{ 1,\dots,N\}$, 
		
		\begin{equation}
			f_i(x,t) = \frac{\s_i(t)}{x-x_i(t)},~ g_i(x,t) = \frac{b_j}{x - v_j t - a_j}.
		\end{equation}
		We have from Lemma \ref{scalar},
		
		\begin{equation}
			\langle f_i,f_j \rangle = \frac{\s_i \bar \s_j}{x_i - \bar x_j} 2 i \pi.
		\end{equation}
		The same method gives
		
		\begin{equation}
			\langle g_i,g_j \rangle = \frac{b_i \bar b_j}{(v_i-v_j) t+a_i - \bar a_j} 2 i \pi,\quad 
			\langle f_i,g_j \rangle = \frac{\s_i \bar b_j}{x_i - v_j t - \bar a_j} 2 i \pi,
		\end{equation}
		
		\begin{equation}
			\text{ and}\quad\langle g_i,f_j \rangle = \frac{b_i \bar \s_j}{v_i t + a_i - \bar x_j} 2 i \pi.
		\end{equation}
	Hence, we obtain that
	
	\begin{multline}
		\langle \sum_{i=1}^N f_i - g_i, \sum_{j=1}^N f_j - g_j \rangle = \sum_{i,j} \langle f_i,f_j \rangle + \sum_{i,j} \langle g_i,g_j \rangle - \sum_{i,j} \langle g_i, f_j \rangle - \sum_{i,j} \langle f_i,g_j \rangle.
	\end{multline}
	
	Since $s_j$ and $b_j$ are bounded, we have for $i\neq j$, $\langle f_i,f_j \rangle \rightarrow 0$, $\langle g_i,g_j \rangle \rightarrow 0$, $\langle f_i,g_j \rangle \rightarrow 0$ and $\langle g_i,f_j \rangle \rightarrow 0$. We also have 
	\begin{equation}
		\langle g_j,g_j \rangle = \frac{|b_j|^2}{\im(a_j)} \pi,~ \langle f_i, f_j \rangle = \frac{|s_i(t)|^2}{\im(x_j)(t)} \pi,~ \langle f_i,g_i \rangle = \frac{\s_i(t) \bar b_i}{x_i(t) - v_i t - \bar a_i} 2 \pi.
	\end{equation}
	
	Now, we have $x_i(t) - v_i t \rightarrow a_i$, $|s_i(t)|^2 \rightarrow |b_j|^2$. Hence, 
	\begin{multline}
		\langle \sum_{i=1}^N f_i - g_i, \sum_{j=1}^N f_j - g_j \rangle = \sum_{i=1}^N \langle f_i,f_i \rangle + \sum_{i=1}^N \langle g_i,g_i \rangle - \sum_{i=1}^N \langle g_i, f_i \rangle - \sum_{i=1}^N \langle f_i,g_i \rangle \\
		+ \epsilon(t) = \sum_{i=1}^N \frac{|s_i(t)|^2}{\im(x_i)(t)} \pi + \sum_{i=1}^N \frac{|b_j|^2}{\im(a_j)} \pi - \sum_{i=1}^N \frac{\s_i(t) \bar b_i}{x_i(t) - v_i t - \bar a_i} 2\pi -  \sum_{i=1}^N \frac{\bar \s_i(t)  b_i}{a_i + v_i t -\bar x_i(t)} 2\pi \\
		+ \epsilon(t) \rightarrow 0.
	\end{multline}
	We now go on with the proof of \eqref{sobo}. For this, we recall
		\begin{equation}\label{formule2bis}
		\langle \frac{\s_i}{(x-x_i)^{k+1}} , \frac{\s_i}{(x-x_i)^{k+1}} \rangle = \frac{|s_i|^2}{\im(x_i)^{2k+1}} \frac{\sqrt{\pi} \Gamma(k+\frac{1}{2})}{\Gamma(k+1)},
	\end{equation}
	
	\begin{equation}\label{formule3bis}
		\langle \frac{\s_i}{(x-x_i)^{k+1}}, \frac{\s_j}{(x-x_j)^{k+1}} \rangle \xrightarrow{t\rightarrow \infty} 0.
	\end{equation}
	The same method gives
	
			\begin{equation}\label{formule4bis}
		\langle \frac{b_i}{(x-z_i)^{k+1}} , \frac{b_i}{(x-z_i)^{k+1}} \rangle = \frac{|b_i|^2}{\im(z_i)^{2k+1}} \frac{\sqrt{\pi} \Gamma(k+\frac{1}{2})}{\Gamma(k+1)},
	\end{equation}
	
	\begin{equation}\label{formule5bis}
		\langle \frac{b_i}{(x-z_i)^{k+1}}, \frac{b_j}{(x-z_j)^{k+1}} \rangle \xrightarrow{t\rightarrow \infty} 0,
	\end{equation}
	and we also have (recall that $z_i = v_i t+ a_i$)
	
	\begin{equation}
		\left( \frac{\partial}{\partial x} \right)^k \left( \frac{b_i}{x-z_i} \right) = (-1)^k k! \frac{b_i}{(x-z_i)^k}.
	\end{equation}
	Now, we need to study the cross terms. We have by definition
	
	\begin{equation}
		\langle \frac{\s_i}{(x-x_i)^k}, \frac{b_i}{(x-z_i)^k} \rangle = \int_{x=-\infty}^\infty \frac{\s_i}{(x-x_i)^k} \frac{\bar b_i}{(x-\bar z_i)^k} dx.
	\end{equation}
	Now, since $x_i - v_i t \rightarrow a_i(\infty)$ with $\im(a_i(\infty)) >0$, for $t$ big enough, we have 
	
	\begin{equation}
		\frac{|s_i|}{|x-x_i|^k} \leq C \min \left( \frac{|s_i|}{|x|^k}, \frac{2}{|\im(a_i(\infty))|} \right),
	\end{equation}
and 
	\begin{equation}
	\frac{|b_i|}{|x-z_i|^k} \leq C \min \left( \frac{|b_i|}{|x|^k}, \frac{2}{|\im(a_i(\infty))|} \right),
	\end{equation}
	which means that we can apply the dominated convergence theorem. Now,
	
	\begin{multline}
		\int_{x=-\infty}^\infty \frac{\s_i}{(x-x_i)^k} \frac{\bar b_i}{(x-\bar z_i)^k} dx =
		\int_{x=-\infty}^\infty \frac{\s_i}{(x-(x_i-v_i t))^k} \frac{\bar b_i}{(x-(\bar z_i-v_i t))^k} dx \\
		\xrightarrow{t\rightarrow \infty} \int_{x=-\infty}^\infty \frac{b_i}{(x-a_i(\infty))^k} \frac{\bar b_i}{(x-a_i(\infty))^k} dx = \frac{|b_i|^2}{\im(a_i)^2} \frac{\sqrt{\pi}\Gamma(k+\frac{1}{2})}{\Gamma(k+1)}.
	\end{multline}
	Finally, for $i \neq j$, we assume without loss of generality that $v_i \leq v_j$, for $t$ big enough,
	
	\begin{multline}
		\langle \frac{\s_i}{(x-x_i)^k}, \frac{b_j}{(x-z_j)^k} \rangle = \int_{x = -\infty}^\infty \frac{\s_i}{(x-x_i)^k} \frac{\bar b_j}{(x-\bar z_j)^k} dx \\
		=\int_{x = -\infty}^{(x_i+\bar z_j)/2} \frac{\s_i}{(x-x_i)^k} \frac{\bar b_j}{(x-\bar z_j)^k} dx + 
		\int_{x = (x_i+\bar z_j)/2}^{\infty} \frac{\s_i}{(x-x_i)^k} \frac{\bar b_j}{(x-\bar z_j)^k} dx 
	\end{multline}
	Now, on $(\infty,(x_i+\bar z_j)/2]$ we have that
	
	\begin{equation}
		|x-\bar z_j|\geq \frac{|\bar z_j - x_i|}{2},
	\end{equation}
	and on $[(x_i+\bar z_j)/2,\infty)$ we have that 
	
	\begin{equation}
		|x-\bar x_i|\geq \frac{|\bar z_j - x_i|}{2}.
	\end{equation}
	Hence, 
	
	\begin{equation}
		\left| \langle \frac{\s_i}{(x-x_i)^k}, \frac{b_j}{(x-z_j)^k} \rangle\right| \leq \frac{2 |s_i| |b_j| }{|\bar z_j - x_i|} \left( \int_{-\infty}^\infty \frac{1}{|x-x_i|^k} dx + \int_{-\infty}^\infty \frac{1}{|x-z_j|^k} \right) dx \rightarrow 0.
 	\end{equation}
	Now, we have 
	
	\begin{multline}
		\left| \left| \left(\frac{\partial}{\partial x}\right)^k (\m(t)-F(t)) \right| \right| = \sum_{i=1}^N \langle \frac{\s_i}{(x-x_i)^k},\frac{\s_i}{(x-x_i)^k} \rangle + \sum_{i=1}^N \langle \frac{b_i}{(x-z_i)^k},\frac{b_i}{(x-z_i)^k} \rangle \\
		-\sum_{i=1}^N \langle \frac{\s_i}{(x-x_i)^k},\frac{b_i}{(x-z_i)^k} \rangle - \sum_{i=1}^N \langle \frac{b_i}{(x-b_i)^k},\frac{\s_i}{(x-x_i)^k} \rangle +\oo(1) \rightarrow 0.
	\end{multline}
    
	\end{proof}

	\subsection{Sufficient conditions for a scattering behavior starting at $t=0$}
	
	In this section, we will establish sufficient conditions for initial data to admit a global in-time and scattering solution. The condition we define intuitively describes to a situation obtained after a long time, but can be checked locally for any time. 
    It corresponds to low interactions between the solitons, and we show that the interactions are not sufficient to reinforce the interactions, which leads to a situation preserved for all times moving forward.

    The interest of this result is twofold. First, we use it in conjunction with a fixpoint argument to create solutions with non-degenerate spectrum for all $N$ number of solitons. We also show that the condition we present has to be satisfied for some time for any solution, and hence obtain the well-posedness and scattering of any solution as long as the speeds are different.

    We denote 
    \[
    \left\{
    \begin{aligned}
        &S=\max_{j} |\s_j(0)|,\\
        &\nu=\min_{j\neq k} | \re(\dot x_j(0)) -  \re(\dot x_k(0))|,\\
        &D=\min_{j\neq k} |\re(x_j(0))- \re(x_k(0))|.
    \end{aligned}
    \right.
    \]
	
	We first present and show the following theorem.
	
	\begin{proposition}\label{scattering}

            Let $s_j(0)$ spins and $x_j(0)$ poles satisfying the constraints, and assume
    \[
    \begin{aligned}
        (i)\quad & \nu > 0, \text{ (different speeds) }\\
        (ii)\quad & v_j > v_k \Rightarrow \re(x_j(0)) > \re (x_k(0)). \text{ (same order as the speeds) } 
    \end{aligned}
    \]
    If there exists $\kappa$, $S'$ and $\eta>0$ satisfying
		\begin{equation}
			\left\{
			\begin{aligned}
                &\nu \geq \frac{4}{\kappa} + \eta,\\
				&S' > S+ \frac{2N S'^2}{ \eta D},\\
				&D^2 \geq \frac{2 N S'^2 \kappa}{ \eta },\\
			\end{aligned}
			\right.
		\end{equation}
		Then, there exists a global, non-turbulent solution for $t>0$ to the half wave map equation, with initial condition $(x_1(0),\dots,x_N(0))$, $(s_1(0),\dots,s_N(0))$. Also, all the following properties are satisfied.
		\begin{itemize}
			\item[($\alpha$)] The poles uniformly move away from each other, i.e. for any $j \neq k$,
			\begin{equation}
				\frac{\partial }{\partial t} |\re(x_j(t)) - \re(x_k(t))| \geq \eta.
			\end{equation}
			\item[($\beta$)] The speeds stay close to the initial speeds, i.e. for any $t>0$ and $j\in \{1,\dots,N\}$
			\begin{equation}
				\re\left( \dot x_j(t) \right) \in \left[ \re\left( \dot x_j(0)\right) - \frac{2}{\kappa}, \re\left(\dot x_j(0)\right)+ \frac{2}{\kappa}  \right].
			\end{equation} 
			\item[($\gamma$)] The spins stay bounded, i,e and we have for any $j$, for any $t>0$,
			\begin{equation}
				|\s_j(t) | \leq S'.
			\end{equation}
			\item[($\delta)$] The imaginary parts of the poles satisfy
    		\begin{equation}
    			\operatorname{Im}(x_j(t)) \geq \frac{|\s_j(0)| }{\sqrt{2}} e^{\frac{-4NS'}{\eta D}} \frac{1}{1+N \frac{S'}{D}}.
    		\end{equation}
		\end{itemize}
	\end{proposition}
	
	To prove the theorem, we will proceed by steps. First, we will assume that the spins stay bounded, and show all the properties except $(\gamma)$. Then, we will show that the spins are bounded. Lastly, we will show $(\delta)$. For the first step, we will use a bootstrap argument. For the second step, we will use the time evolution equations, again with a bootstrap argument. For the last step, we will use the fact that $\m$ is of constant norm $1$, and evaluate $\m$ at well-chosen points, i.e. $z_j = Re(x_j(t))$.

	\begin{proof}(\textbf{Of ($\alpha$) and ($\beta$), assuming that the spins are bounded})
		
		We will use a bootstrap argument and assume that the spins are bounded by $S'$. We define for $T>0$,
		
		\begin{equation}
			\textbf{Property 1: for t $\in [0, T ]$, } \re\left( \dot x_j(t) \right) \in \left[ \re \left(\dot x_j(0)\right) - \frac{2}{\kappa} , \re \left(\dot x_j(0)\right) + \frac{2}{\kappa} \right],
		\end{equation}
		and 
		
		\begin{equation}
			\textbf{Property 2: for t $\in [0, T ]$, } \re\left( \dot x_j(t) \right) \in \left[ \re \left(\dot x_j(0)\right) - \frac{1}{\kappa} , \re \left(\dot x_j(0)\right) + \frac{1}{\kappa} \right].
		\end{equation}
		We will show that assuming Property 1 leads to Property 2, and we will hence obtain property 1.
		
		We now assume \textbf{Property 1}. The property 
		
		\begin{equation}
			d \left( \left[ \re \left(\dot x_j(0)\right) - \frac{2}{\kappa} , \re \left(\dot x_j(0)\right) + \frac{2}{\kappa} \right],~ \left[ \re \left(\dot x_k(0)\right) - \frac{2}{\kappa} , \re \left(\dot x_k(0)\right) + \frac{2}{\kappa} \right] \right) \geq \eta
		\end{equation}
		is satisfied for any $j\neq k$, as long as we have for any $j \neq k$,

        \[
        \nu \geq \frac{4}{\kappa} + \eta,
        \]
        which is satisfied by assumption. From this, we obtain
        \[
        |\dot x_j(t) - \dot x_k(t)| \geq |\re(\dot x_j(t)) - \re (\dot x_k(t))| \geq \eta,
        \]
		This implies in particular for any $t \in [0,T]$, $|x_j(t) - x_k(t)| \geq D + t \eta$. The time evolution equation for $x_j$ writes
		\begin{equation}
			\poleTE{s}{x}
		\end{equation}
		Hence, we can write
		
		\begin{equation}
			\frac{\partial }{\partial t} | \dot x_j(t)| \leq 4 \frac{N S'^2}{ (|x_j(0) - x_k(0)| + t \eta)^3 }.
		\end{equation}
		In particular, we obtain
		\begin{equation}
			|\dot x_j(t) - \dot x_j(0)| \leq 4 N S'^2 \int_{z=0}^t \frac{1}{(D+z \eta)^3} dz = \frac{2 N S'^2}{ \eta D^2} - \frac{2 N S'^2}{ \eta (D^2 + t\eta)^2} \leq \frac{2 N S'^2}{\eta D^2}.
		\end{equation}
		Now, since by assumption, the inequality  
		\begin{equation}
			\frac{2NS'^2}{ \eta D^2 } \leq \frac{1}{\kappa}
		\end{equation}
		is satisfied, we obtain \textbf{Property 2}. Hence, Property 1 holds for any $T>0$.

	\end{proof}
	
	We now go on, and show that the spins are bounded.
	
	\begin{proof}(\textbf{That the spins are bounded})
	We will now show that the spins are bounded by a bootstrap argument. We define the properties
		
		\begin{equation}
			\textbf{Property 1: for t $\in [0, T ]$, } |\dot \s_j(t)| \leq S',
		\end{equation}
		and 
		
		\begin{equation}
			\textbf{Property 2: for t $\in [0, T ]$, } |\dot \s_j(t)| \leq S'',
		\end{equation}
		for some $S'' < S'$. We will show that assuming Property 1 leads to Property 2, and Property 1 will then be satisfied.
		
		We hence assume Property 1. Since the spins are now by assumption bounded by $S'$, we can hence use $(\alpha)$ and $(\beta)$.
		
		We recall the evolution equations for the spins
		
		\begin{equation}
			\spinTE{s}{x},
		\end{equation}	
		for $j \in \{1,\dots, N\}$. Using $(\alpha)$, $(\beta)$ and \textbf{Property 1}, 
	
		\begin{equation}
			\frac{\partial }{\partial t} |\s_j(t)| \leq \frac{2 N S'^2}{ (D+ \eta t )^2 }.
		\end{equation}
		Hence,
		
		\begin{multline}
			\left| |\s_j(t)| - |\s_j(0)| \right| \leq 2 N S'^2  \int_{z=0}^t  \frac{1}{(D+ \eta z)^2} \dd z  \leq \frac{2 N S'^2}{\eta^2} \left( \frac{1}{D/\eta} - \frac{1}{(D/\eta+ t)} \right) \\
			\leq \frac{2 N S'^2}{ \eta D }.
		\end{multline}
		since $|\s_j(0)| \leq S$, we obtain
		
		\begin{equation}
			|\s(t)| \leq |\s_j(0)| + \frac{2 N S'^2}{\eta D} \leq S + \frac{2 N S'^2}{\eta D} < S'' := \frac{1}{2} S' + \frac{1}{2} \left( S + \frac{2 N S'^2}{\eta D} \right),
		\end{equation}
        where the last inequality is strict since $S + \frac{2 N S'^2}{\eta D}<S'$. This concludes the bootstrap argument, so we have $|\s_j(t)| \leq S'' \leq S'$.
	\end{proof}
	
	We hence now know that under the aforementioned assumptions, properties $(\alpha)$, $(\beta)$, and $(\gamma)$ are satisfied. We now go on with $(\delta)$.
	
	\begin{proof}(\textbf{Of ($\delta$)})
		
		We will use the fact that the function
		
		\begin{equation}
			\m(t,x) = \m_0 + \iu \sum_{j=1}^N \frac{\s_j(t)}{x-x_j(t)} - \iu \sum_{j = 1}^N \frac{\bar \s_j(t)}{ x - \bar x_j(t)},
		\end{equation}
		stay bounded (in norm) by 1. We will use the fact that
		
		\begin{equation}
			|x_j(t) - x_k(t)| \geq |\re(x_j(t)) - \re(x_k(t))|  \geq  |\tilde x_j(0) - \tilde x_k(0)| + t \eta.
		\end{equation}
		We choose $x=\tilde x_j(t) = \re(x_j(t))$ and obtain
		
		\begin{equation}
			m(\tilde x_j(t)) = \m_0 + \iu \sum_{k\neq j} \frac{\s_k(t)}{ \tilde x_j(t) - x_k(t)} - \iu \sum_{k \neq j} \frac{\bar \s_j(t)}{ \tilde x_j(t) - \bar x_k(t)} + 2 \re\left( \iu \frac{\s_j(t)}{- \im (x_j(t))} \right).
		\end{equation}
		In particular, we obtain
		
		\begin{equation}
			\left| 2 \re\left( \frac{- i \s_j(t)}{ \im(x_j(t))} \right) \right| \leq 1 + 1 + 2 N \frac{ S'}{ D } \leq C.
		\end{equation}
		Now, 
		
		\begin{equation}
			\left| 2 \re\left( \frac{- \iu \s_j(t)}{ \im(x_j(t))} \right) \right| = 2 \left| \im\left( \frac{ \s_j(t)}{ \im(x_j(t))} \right) \right| = \frac{2}{|\im(x_j(t))|} \left| \im(s_j(t)) \right|.
		\end{equation}
		Since we have the compatibility condition $s_j(t)^2 = 0$, which yields in particular
		
		\begin{equation}
			\left\{
			\begin{aligned}
				&|\re(s_j(t))|^2 = |\im(s_j(t))|^2,\\
				&\langle \re(s_j(t)), \im(s_j(t)) \rangle =0,
			\end{aligned}
			\right.
		\end{equation}
		and hence $|\im(\s_j(t))| = \frac{|\s_j(t)|}{\sqrt{2}}$. We now show that the norm of the spins are bounded from below. The time evolution equations of the spins imply 
		
		\begin{equation}
			\left| \frac{\partial }{\partial t} |\s_j(t) | \right| \leq \frac{4N  S' |\s_j(t)|}{ (D+ \eta t)^2 }.
		\end{equation}
		
		In particular, we have
		
		\begin{equation}
			\left| \ln \left( \frac{|\s_j(t)|}{|\s_j(0)|} \right) \right| \leq \frac{4 N S'}{ \eta D },
		\end{equation}
		so 
		\begin{equation}
			|\s_j(t)| \geq |\s_j(0)| e^{\frac{-4NS'}{\eta D}}.
		\end{equation}
		Hence, we obtain
		\begin{equation}
			\im(x_j(t)) \geq \frac{|\s_j(0)| }{\sqrt{2}} e^{\frac{-4NS'}{\eta D}} \frac{1}{1+N \frac{S'}{D}}.
		\end{equation}
		
	\end{proof}

    \begin{corollary}
        Let $\m^0(x) $ be a rational function defined as 
    \[
    \m^0(x) = \m_0 + \sum_{j=1}^N \frac{\s_j(0)}{x-x_j(0)} + \sum_{j=1}^N \frac{\bar \s_j(0)}{x - \bar x_j(0)},
    \]
    satisfying the constraints and define $\alpha_0$ as 
    \[
    \alpha_0 = D - \frac{2^4 N S}{\nu}.
    \]
    Then, if $\alpha_0 > 0$, the Cauchy problem 
    \[
    \left\{
    \begin{aligned}
        &\partial_t \m(t,x) = \m(t,x) \times |\nabla| \m(t,x),~ t>0 \\
        &\m(0,x) = \m^0(x).
    \end{aligned}
    \right.
    \]
    admits a global in time solution for $t>0$, instantly scattering in the sense that its time evolution is described by Proposition \ref{scattering}. 

    Moreover, for $\m(t,x)$ a solution of $(HWM)$ with spins $s_j(t)$ and poles $x_j(t)$, if there exists $t_0$ such that 
    \[
    \alpha(t_0) = D(t_0) - \frac{2^4 N S(t_0)}{\nu(t_0)x} > 0,
    \]
    the solution satisfies the assumption of Proposition \ref{scattering} and is instantly scattering.

    \end{corollary}

    \begin{proof}
        It suffices to show that assuming 
        \[
        \alpha_0 =  D - \frac{2^4 N S}{\nu} > 0 
        \]
        provides the existence of $\kappa$, $S'$, $\eta$ such that the systen
        \[
        \left\{
        \begin{aligned}
            &\nu \geq \frac{4}{\kappa} + \eta,\\
            &S' > S + \frac{2N S'^2}{\eta D},\\
            &D^2 \geq \frac{2N S'^2 \kappa}{\eta}, 
        \end{aligned}
        \right.
        \]
        is satisfied. We choose $S' = 2S$, $\eta = \nu/2$, $\kappa = 8/\nu$, and remark that the system is satisfied. Indeed, the first line leads to $\nu = \nu$, and the second and third line respectively lead to
        \[
        D > \frac{16 N S}{\nu},\quad \text{ and } D^2 \geq \frac{2^7 N S^2}{\nu^2}
        \]
        which are satisfied by assumption.
        
    \end{proof}

    \section{Creation of $N$-solitons solutions with non-degenerate spectrum}\label{sec:creation}

\subsection{Defining the approximate initial condition}

In this subsection, we create an approximate initial condition. More precisely, our goal is to define an initial condition $\m(0,x)$, such that $\m$ is a valid initial condition and that it satisfies the instantaneous scattering condition provided by Theorem \ref{thm:localCondition}. Ideally, we would also like $\m$ to have the spectrum $v_1,\dots,v_N$ of our choice. 
Here, we first define an initial condition that has the initial speeds of our choice, and that almost satisfies the constraints. The next lemma precisely states what this function satisfies.

\begin{lemma}
    Let $w_1,\dots, w_N$ be a non-degenerate spectrum, i.e. $|w_i - w_j| \geq \nu >0$ for $i \neq j$. For any $\varepsilon > 0$, there exists an initial condition $\m(0,x)$ such that 

    \begin{equation}
    \left\{
        \begin{aligned}
            &|S_0| = \max_j |\s_j(0)| \leq 10,\\
            &\left| \dot x_j(0) - w_j \right| \leq \varepsilon,~ \forall j, \\
            &s_j^2 = 0,~ \forall j,\\
            &\left| s_j \cdot \left( i \m_0 - \sum_{k\neq j}^n \frac{\s_k}{x_j-x_k} + \sum_{k=1}^N \frac{\bar \s_k}{x_j - \bar x_k} \right) \right| \leq \varepsilon,~ \forall j.
        \end{aligned}
        \right.
    \end{equation}

\end{lemma}

\begin{proof}

        Let $w_1,\dots,w_N$ be a non-degenerate spectrum. We start by choosing a set of speeds $\dot x_j(0) = w_j$. We consider $D \in \mathbb{R}$ and choose 
$x_j(0) = D \cdot \dot x_j(0)$.
    In order to for a function to be a valid initial condition, the spins and poles have to satisfy 

	\begin{equation}
		\left\{
		\begin{aligned}
			&s_j(0)^2 = 0,\\
			&\dot x_j(0) = \frac{\s_j(0) \times \bar \s_j(0)}{ |\s_j(0)|^2 } \cdot (i \m_0 - \sum_{k \neq j} \frac{\s_k(0)}{ x_j(0) - x_k(0) } + \sum_{j \neq k} \frac{\bar \s_k(0)}{ x_j(0) - \bar x_k(0) }  ),\\
			&s_j(0) \cdot \left( i \m_0 - \sum_{k \neq j} \frac{\s_k(0)}{x_j(0)- x_k(0)} + \sum_{k \neq j} \frac{\bar \s_k(0)}{ x_j(0) - \bar x_k(0)} \right)=0.
		\end{aligned}
		\right.
	\end{equation}

    We will solve the compatibility conditions by assuming $|x_k(0) - x_j(0)| = \infty$ (so a simplified version), and obtain values for the spins and the imaginary parts. Now, this is "almost" satisfying the compatibility condition. We will use a fixpoint argument to justify that we can find an initial data that satisfies the compatibility condition, while being close to our approximate initial data. 

    We consider the approximate conditions 
    \begin{equation}\label{condapprox}
    	\left\{
    	\begin{aligned}
    		&\dot x_j(0) = \frac{\s_j(0) \times \bar \s_j(0)}{|\s_j(0)|^2} \cdot (i \m_0),\hspace{1.25cm} (i)\\
    		&\s_j(0) \cdot \left(i \m_0 + \frac{\bar \s_j(0)}{2 i \im(x_j(0))} \right) =0,\hspace{0.5cm} (ii)\\
    		&\s_j(0)^2 = 0.\hspace{4.1cm} (iii)
    	\end{aligned}
    	\right.
    \end{equation}
    We denote
	
	\begin{equation}
		\s_j(0) = \begin{bmatrix} \alpha_j \\  \beta_j \\ \gamma_j \end{bmatrix}, \text{ and } \m_0 = \begin{bmatrix} 0 \\  0 \\ 1 \end{bmatrix}.
	\end{equation}
	Now, (ii) of \eqref{condapprox} gives
	
	\begin{equation}
		\iu \s_j(0) \cdot \m_0 - \frac{\iu |\s_j(0)|^2}{2 \im(x_j(0))}=0,
	\end{equation}
	so
	
	\begin{equation}
		\gamma_j = \frac{|\s_j(0)|^2}{2 \im(x_j(0))} \in \mathbb{R}.
	\end{equation}
	We need $\s_j(0)^2=0$, which gives $\alpha_j^2 + \beta_j^2 + \gamma_j^2 = 0$. We will chose $\alpha_j = 2 e^{i \theta_1^j}$ and $\beta_j = 2 e^{i \theta_2^j}$. Hence, the compatibility condition gives
	
	\begin{equation}
		-\gamma_j^2 = 4 ( e^{\iu 2 \theta_1^j } + e^{\iu 2 \theta_2^j}) = 8 e^{2 \iu \frac{\theta_1^j+\theta_2^j}{2}} \cos(\theta_1^j - \theta_2^j).
	\end{equation}
	This implies that $\theta_1^j+\theta_2^j = 0[\pi]$. We will take $\theta_1^j + \theta_2^j =0$, and we obtain 

	\begin{equation}
		\gamma_j^2 = 8 \cos(\theta_1^j - \theta_2^j). 
	\end{equation}	
	Now, we look at (i) of \eqref{condapprox}. Since we have $|\s_j(0)|^2 = 4 + 4 + 8 \cos(\theta_1^j - \theta_2^j)$
	
	\begin{equation}
		\dot x_j(0) = \frac{\s_j(0) \times \bar \s_j(0)}{8 (1 + \cos(\theta_1^j + \theta_2^j) )} \cdot i \m_0 = i \frac{\alpha_j \bar \beta_j - \bar \alpha_j \beta_j}{8 (1+\cos(\theta_1^j - \theta_2^j))} = \frac{- 8 \sin(\theta_1^j  - \theta_2^j)}{8 + \delta_j^2}.
	\end{equation}
	Since we expect that $\dot x_j(\infty) = v_j$ is going to be very close to $\dot x_j(0)$, in order to chose $v_j$, we modify $\delta_j$. In fact, one can write
	
	\begin{equation}
		\dot x_j^2(0) = \frac{8 - \delta_j^2}{ 8 + \delta_j^2 },
	\end{equation}
	and $v_j \in \left[ \frac{N-2}{N} \dot x_j(0), \frac{N+2}{N} \dot x_j(0) \right]$, if one chooses the parameters accordingly. Inverting the relationship yields
	
	\begin{equation}
		\delta_j^2 = \frac{8 (1-\dot x_j(0)^2)}{1+\dot x_j(0)^2}.
	\end{equation}x
	Hence, we define 
 
	\fbox{
		\addtolength{\linewidth}{-2\fboxsep}%
		\addtolength{\linewidth}{-2\fboxrule}%
		 \begin{minipage}{\linewidth}	
	\begin{equation}\label{IC}\tag{IC}
	\s_j(0) = 	\begin{bmatrix} \alpha_j \\  \beta_j \\ \gamma_j \end{bmatrix} = \begin{bmatrix} 2 e^{\iu \theta_1^j} \\  2 e^{ - \iu \theta_1^j} \\ 2 \sqrt{ 2 |\cos(2 \theta_1^j)| } \end{bmatrix},~ \theta_1^j = \frac{1}{2} \cos^{-1} \left( \frac{1-\dot x_j(0)^2}{1+\dot x_j(0)^2} \right),
	\end{equation}
 \end{minipage}
}
    and we have to choose 

    \fbox{
		\addtolength{\linewidth}{-2\fboxsep}%
		\addtolength{\linewidth}{-2\fboxrule}%
		\begin{minipage}{\linewidth}
	\begin{equation}
		\im(x_j(0)) = \frac{|\s_j(0)|^2}{2 \gamma_j} = \frac{8 ( 1 + \cos(2 \theta_1^j)^2)}{2 \gamma_j}.
	\end{equation}
\end{minipage}
}
Now, the real speeds $\dot x_j(0)$ satisfy

\begin{multline}
    \dot x_j(0) = \frac{\s_j(0)\times \bar \s_j(0)}{|\s_j(0)|^2} \cdot \iu \m_0 \\ 
    + \frac{\s_j(0)\times \bar \s_j(0)}{|\s_j(0)|^2} \cdot \left( - \sum_{k\neq j} \frac{\s_k(0)}{x_j(0)-x_k(0)} + \sum_{j\neq k} \frac{\bar \s_k(0)}{x_j(0) - \bar x_k(0)} \right) \\ 
    = w_j + \frac{\s_j(0)\times \bar \s_j(0)}{|\s_j(0)|^2} \cdot \left( - \sum_{k\neq j} \frac{\s_k(0)}{x_j(0)-x_k(0)} + \sum_{j\neq k} \frac{\bar \s_k(0)}{x_j(0) - \bar x_k(0)} \right).
\end{multline}
So the distance between the desired spectrum and the initial speeds satisfy 

\begin{equation}
    |\dot x_j(0) - w_j | \leq \frac{2(N-1) |S_0|}{\nu D} \leq \varepsilon,
\end{equation}
for $D \geq \frac{2 (N-1) |S_0|}{\nu \varepsilon}$. Finally, the constraints are indeed almost satisfied, since $s_j^2 =0$ and 

\begin{multline}
    \left| \s_j(0) \cdot \left( \iu \m_0 - \sum_{k\neq j} \frac{\s_k(0)}{x_j(0)-x_k(0)} + \sum_{k = 1}^N \frac{\bar \s_k(0)}{x_j(0) - \bar x_k(0)} \right) \right| \\ 
    = \left| \s_j(0) \cdot \left( - \sum_{k\neq j} \frac{\s_k(0)}{x_j(0)-x_k(0)} + \sum_{k \neq j} \frac{\bar \s_k(0)}{x_j(0) - \bar x_k(0)} \right) \right| \leq \frac{2(N-1) |S_0|^2}{\nu D} \leq \varepsilon,
\end{multline}
for $D \geq \frac{2(N-1)|S_0|^2}{\nu \varepsilon}$.

\end{proof}

\subsection{Inductive definition of the sequence}

We define $\mathcal{S}_0 = (\s_1(0),\dots,\s_N(0))$, and for $\HH = (\h_1,\dots,\h_N) \in \left(\mathbb{C}^{3}\right)^N$ we define the three functionals 

\begin{equation}
    F(\HH)_j = \m_0 - \sum_{k\neq j} \frac{\h_k}{x_j-x_k} + \sum_{k\neq j} \frac{\bar \h_k}{x_j - \bar x_k} + \frac{\bar \h_j}{2 \iu \im(x_j)},
\end{equation}

\begin{equation}
    \tilde F(\HH)_j = - \sum_{k\neq j} \frac{\h_k}{x_j-x_k} + \sum_{k\neq j} \frac{\bar \h_k}{x_j - \bar x_k} + \frac{\bar \h_j}{2 \iu \im(x_j)},
\end{equation}

\begin{equation}
    \vardbtilde F(\HH)_j = \m_0 - \sum_{k\neq j} \frac{\h_k}{x_j-x_k} + \sum_{k\neq j} \frac{\bar \h_k}{x_j - \bar x_k} + \frac{\bar \h_j}{2 \iu \im(x_j)}.
\end{equation}

Now, the constraints can be rewritten using the mapping $F$. Indeed, $S=(\s_1,\dots,\s_N)$ satisfies the constraints if and only if

\begin{equation}
    S \cdot F(S) = 0,~ S^2 =0,
\end{equation}

where $S \cdot F(S) =0$ and $S^2 = 0$ are two systems of $N$ equations given by 

\begin{equation}
    \begin{aligned}
        &\langle s_j\cdot F(S)_j \rangle_{\mathbb{R}^3} =0,~\forall j \in \{1,\dots,N \},\\
        &\langle s_j \cdot s_j \rangle_{\mathbb{R}^3} = 0,~\forall j\in \{ 1,\dots,N\}.
    \end{aligned}
\end{equation}

In our situation, the biggest order term is due to $\m_0$ in the expression of $F$ and is of order $1$. Then, the biggest order term in $\tilde F$, which is $F$ without the term in $\m_0$, is $\frac{\bar \h_j}{\im(x_j)}$, and is of order $|\h_j|$. Finally, the other terms are of order $|\h_j|/D$.

For $S_0$ and at any step $S_l$, we have $S_l^2=0$ but $S_l \cdot F(S_l) \neq 0$. Our goal is to define $S_{l+1}$ so that this condition is closer to being satisfied. At every step, we thus define $T_l$ so that $(S_l + T_l) \cdot F(S_l+ T_l)$ is close to zero, and $H_l$ so that $S_{l+1}=S_l+T_l+H_l$ satisfies $S_{l+1}^2=0$. We also want that $H_l$ does not interfere with the improvement from $T_l$ for the first condition. 

\textbf{Definition of $T_l$ and $H_l$ from $S_l$:}

We consider the condition

\begin{equation}
    (S_l+T_l) \cdot F(S_l+T_l) = 0,
\end{equation}
which gives
\begin{equation}
    S_l \cdot F(S_l) + S_l \cdot \tilde F(T_l) + T_l \cdot F(S_l) + T_l \cdot \tilde F(T_l) =0.
\end{equation}

We will neglect the last term and solve a simplified version of $S_l \cdot F(S_l) + S_l \cdot \tilde F(T_l) + T_l \cdot F(S_l) =0$. This equation can be written for any $j$ as

	\begin{multline}\label{ancienne}
		\s_j^l \cdot \left(\iu\m_0 - \sum_{k\neq j} \frac{\s_k^l}{x_j-x_k} + \sum_{k \neq j} \frac{\bar \s_k^l}{x_j - \bar x_k} - \frac{i \bar \s_j^l}{2 \im(x_j)} \right) \\
		+ 	\s_j^l \cdot \left( - \sum_{k\neq j} \frac{\mathbf{t}_k^l}{x_j-x_k} + \sum_{k \neq j} \frac{ \bar{\mathbf{t}}_k^l}{x_j - \bar x_k} - \frac{\iu \bar{\mathbf{t}}_j^l}{2 \im(x_j)} \right) \\ 
		+ \mathbf{t}_j^l \cdot \left(\iu\m_0 - \sum_{k\neq j} \frac{\s_k^l}{x_j-x_k} + \sum_{k \neq j} \frac{\bar \s_k^l}{x_j - \bar x_k} - \frac{\iu \bar \s_j^l}{2 \im(x_j)} \right) = 0.
	\end{multline}

	We will solve instead the equation
	
	\begin{equation}\label{vraieequationT}
		\s_j^l \cdot F_j(S_l) - \frac{\iu \s_j^l \cdot \bar{ \mathbf{t}}_j^l}{2 \im(x_j) } + \iu \mathbf{t}_j^l \cdot \m_0 - \frac{\iu \mathbf{t}_j^l \cdot \bar{\s}_j^l}{2 \im(x_j)}=0.
	\end{equation}

    We take $\mathbf{t}_j^l$ of the form 
	
	\begin{equation}\label{formeT}
		\mathbf{t}_j^l = (0,0,\kappa_j^l),~ \kappa_j^l \in \mathbb{C}.
	\end{equation}
	
	Solving \eqref{vraieequationT} using \eqref{formeT} gives $\re(\kappa_j^l) = \im(\s_j^l \cdot F_j(S_l))$,
	and 
	
	\begin{equation}
		\im(\kappa_j^l) = \frac{\im(x_j) \re(\s_j^l \cdot F_j(S_l))}{ \re(\s_j^l[3])-1 } - \re(\kappa_j^l) \im(\s_j^l[3]).
	\end{equation}

This define $T_l$. Note that $T_l$ satisfies (as long as $|S_l-S_0| \leq |S_0|/2$)

\begin{equation}
    \boxed{|T_l| \leq (K+|S_l|) |S_l \cdot F(S_l)|},
\end{equation}

and plugging \eqref{vraieequationT} into \eqref{ancienne} gives

\begin{equation}
    \boxed{|(S_l+T_l)\cdot F(S_l+T_l)| \leq \frac{4N |S_l|(K+|S_l|)}{D\nu} |S_l\cdot F(S_l)|}.
\end{equation}

We now define $H_l$. We consider the two conditions

\begin{equation}
		(S_l + T_l + H_l)^2 =0,~ H_l \cdot F(S_l+t_l) =0.
	\end{equation}
	
	This means that we want for any $j$,
	
	\begin{equation}
		(\s_j^l+{\mathbf{t}}_j^l + \h_j^l)^2 = 0,
	\end{equation}
	
	and 
	
	\begin{equation}
		\h_j^l \cdot \left( \iu \m_0 - \sum_{k\neq j} \frac{\s_j^l + \mathbf{t}_j^l}{x_j-x_k} + \sum_{k\neq j} \frac{\bar \s_j^l + \mathbf{t}_j^l}{x_j - \bar x_k} - \frac{ \iu \bar \s_j^l + \iu \bar{ \mathbf{t}}_j^l}{2 \im(x_j)} \right)=0.
	\end{equation}
	
	We consider $\mathbf{k}_j^l$ a unitary vector of the form $(k[1],k[2],0)$ that is orthogonal to $F(\s_j^l + \mathbf{t}_j^l)$, and $\h_j^l = p_j^l \mathbf{k}_j^l$, with $p_j^l \in \mathbb{C}$. We write $(\s_j^l+\mathbf{t}_j^l+p_j^l \mathbf{k}_j^l)^2 =0$, which is equivalent to
	
	\begin{equation}
		(\mathbf{t}_j^l)^2 + 2 (\s_j^l \mathbf{t}_j^l) + (p_j^l \mathbf{k}_j^l)^2 + 2 (p_j^l \mathbf{k}_j^l \s_j^l) + 2 (p_j^l \mathbf{k}_j^l) =0.
	\end{equation}
	
	Now, this is a complex equality, that we can solve for $p_j^l$,
	
	\begin{equation}
		\Delta = (2 \mathbf{k}_j^l \s_j^l + 2 \mathbf{k}_j^l \mathbf{t}_j^l)^2 - 4 (\mathbf{k}_j^l)^2 ((\mathbf{t}_j^l)^2 + 2 (\s_j^l \mathbf{t}_j^l) ),
	\end{equation}
	
	we obtain two possible solutions
	
	\begin{equation}
		p_j^l = \frac{-(2 \mathbf{k}_j^l \s_j^l + 2 \mathbf{k}_j^l \mathbf{t}_j^l) \pm \sqrt{\Delta}}{2 (\mathbf{k}_j^l)^2}.
	\end{equation}
	
	We choose the $+$ solution because this is the one that will give the smallest $p_j^l$. 
	
	We obtain
	
	\begin{multline}
		\left| -(2 \mathbf{k}_j^l \s_j^l + 2 \mathbf{k}_j^l \mathbf{t}_j^l) + \sqrt{\Delta} \right| \\ \leq \left| -(2 \mathbf{k}_j^l \s_j^l + 2 \mathbf{k}_j^l \mathbf{t}_j^l) \left(1 - \sqrt{1- \frac{4 (\mathbf{k}_j^l)^2 ((\mathbf{t}_j^l)^2 + 2 (\s_j^l \mathbf{t}_j^l) ) }{ 4 (\mathbf{k}_j^l \s_j^l + \mathbf{k}_j^l \mathbf{t}_j^l)^2 }} \right) \right| \\
		\leq \left| (2 \mathbf{k}_j^l (\s_j^l + \mathbf{t}_j^l))\frac{4 (\mathbf{k}_j^l)^2 ((\mathbf{t}_j^l)^2 + 2 (\s_j^l \mathbf{t}_j^l) )}{4 ( \mathbf{k}_j^l \s_j^l + \mathbf{k}_j^l \mathbf{t}_j^l)^2} \right| \\
		\leq \frac{C}{2} \left| \frac{(\mathbf{k}_j^l)^2 (\mathbf{t}_j^l)^2   }{\mathbf{k}_j^l \mathbf{t}_j^l} + \frac{\mathbf{k}_j^l \s_j^l \mathbf{k}_j^l \mathbf{t}_j^l}{\mathbf{k}_j^l \s_j^l} \right| \leq \frac{C}{D^l}
	\end{multline}
	
	Overall, we obtain 
	
	\begin{equation}
		\boxed{|H_l|\leq 2 |T_l|\cdot \frac{|T_l+2 S_l|}{|T_l+S_l|}},
	\end{equation}
	
	and with $S_{l+1} = S_l + T_l + H_l$,
	
	\begin{equation}
		\boxed{\s_{l+1}^2 = (S_l+T_l+H_l)^2 = 0}.
	\end{equation}

\subsection{Growth estimates for the soundness of the fixpoint}

Let $S_0$ be defined as in \eqref{IC}. We assume that the sequences $T_l$, $H_l$ and $S_l$ are defined inductively, such that

\begin{equation}
\left\{
    \begin{aligned}
    &|T_l| \leq (K+|S_l)| |S_l \cdot F(S_l)|,\\
    &|(S_l+T_l) \cdot F(S_l + T_l)| \leq \frac{4N |S_l| (K+|S_l|)}{D \nu} |S_l \cdot F(S_l)|,
    \end{aligned}
\right.
\end{equation}

\begin{equation}
    \left\{
    \begin{aligned}
        &|H_l| \leq 2 |T_l| \cdot \left| \frac{T_l + 2 S_l}{T_l + S_l} \right|,\\
        &(S_l + T_l + H_l)^2 =0 ,\\
    \end{aligned}
    \right.
\end{equation}

and finally with $S_{l+1} = S_l + T_l + H_l$,

\begin{multline}
    \left| S_{l+1} \cdot F(S_{l+1}) \right| \leq \frac{4N |S_l| (K+|S_l|) |S_l \cdot F(S_l)|}{D \nu} \\
    + (2N+2)^2 C^2 (K+|S_l|)^2 |S_l \cdot F(S_l)|^2 + \frac{2(N-1)}{D\nu} (|T_l|+|S_l|) \cdot |H_l|.
\end{multline}

Note that by the choice of $S_0$, we have 

\begin{equation}
    |S_0 \cdot F(S_0)| \leq \frac{16N}{D \nu}.
\end{equation}

Consequently, for a fixed initial condition, and so a fixed $\nu,$ $N$, $\im(x_i)$, for any $\varepsilon>0$, there exists $D_0$ such that for $D \geq D_0$,

\begin{equation}
    \left| (S_l)_j[i] - (S_0)_j[i] \right| \leq \varepsilon |(S_0)_j[i]|.
\end{equation}

We show this now by induction, together with the property

\begin{equation}
    \exists C,C_0>0,~ |S_l \cdot F(S_l) | \leq \frac{C_0 \cdot C^{l}}{D^{l+1}},
\end{equation}

where $C_0$ and $C$ are uniform with respect to $D$.

First, $|S_0 \cdot F(S_0)| \leq \frac{16N}{D\nu}$, so 

\begin{equation}
    |T_0| \leq (K+|S_0|) |S_0 \cdot F(S_0)| \leq \frac{16N (K+|S_0|)}{D \nu},
\end{equation}

and 

\begin{equation}
    |H_0| \leq 2 |T_0| \cdot \left| \frac{T_0+2 S_0}{T_0 + S_0} \right|,
\end{equation}

so for $D\geq D_1$ big enough, 

\begin{equation}
    |H_0| + |T_0| \leq \min_j \min_i \varepsilon |(S_0)_j[i]|.
\end{equation}

Hence, $|S_1| \leq (1+\varepsilon)|S_0|$, so 

\begin{equation}
    |S_1 \cdot F(S_1) | \leq \frac{C_0\cdot C}{D^2}, 
\end{equation}

for $D\geq D_2$.

We now assume that the property is satisfied for $k \leq l$, and define $T_{l}$ and $H_l$. The conditions become 

\begin{equation}
    |T_l| \leq (K+(1+\varepsilon)|S_0|) \frac{C_0 C_1^{l}}{D^{l+1}},
\end{equation}
and $|H_l| \leq 4 |T_l|$. We write (using $K \leq 2$, which we will see later)

\begin{multline}
    |S_{l+1} - S_0| \leq \sum_{k=0}^l |S_{k+1}-S_k| \leq \sum_{k=0}^l \left( |H_k| + |T_k| \right) \\
    \leq \sum_{k=0}^l 5 (K+(1+\varepsilon)|S_0|) \frac{C_0 C_1^{k+1}}{D^{k+1}} \leq 50 C_0 \frac{C_1}{D} \frac{1}{1-\frac{C_1}{D}}.
\end{multline}

For $D_3\geq 100 C_0 C_1 |S_0| / \varepsilon$ (not depending on $l$), $|S_{l+1}-S_l| \leq \varepsilon |S_0|$.

Then, 

\begin{multline}
    |S_{l+1} \cdot F(S_{l+1})| \leq \Bigg( \frac{8N |S_0| (K+2|S_0|) }{D\nu} \\+ (2N+2)^2 C^2 (K+2 |S_0|^2) |S_l \cdot F(S_l)|  \Bigg) |S_l \cdot F(S_{l})| + 8 \frac{2(N-1)}{D\nu}  |S_0| |H_l|,
\end{multline}
now

\begin{equation}
    \left| \frac{8N |S_0| (K+2|S_0|) }{D\nu} \right| \leq \left( \frac{8N|S_0|(K+2|S_0|)}{\nu} \right) \frac{1}{D},
\end{equation}

\begin{multline}
    \left| (2N+2)^2 C^2 (K+2|S_0|^2) S_l \cdot F(S_l) \right| \leq \left(  \frac{(2N+2)^2 C^2 (K+2|S_0|^2)C_0 C_1^l}{D^l} \right) \frac{1}{D} \\
    \leq \left(  (2N+2)^2 C^2 (K+2|S_0|^2)C_0 \right) \frac{1}{D} 
\end{multline}
and

\begin{equation}
    \left| 8 \frac{2(N-1)}{D\nu} |S_0| |H_l| \right| \leq \left( \frac{64 (N-1) |S_0| (K+2 |S_0|)C_0}{ \nu} \right) \frac{1}{D} \frac{C_1^l}{D^{l+1}}.
\end{equation}
For $C_1$ satisfying

\begin{multline}
    C_1 \geq \left( \frac{8N|S_0|(K+2|S_0|)}{\nu} \right) + \left(  (2N+2)^2 C^2 (K+2|S_0|^2)C_0 \right) + \\ \left( \frac{16 (N-1) |S_0| (K+2 |S_0|)C_0}{\nu} \right),
\end{multline}
    we indeed have 

    \begin{equation}
        |S_{l+1} \cdot F(S_{l+1})| \leq \frac{C_0 C_1^{l+1}}{D^{l+2}}.
    \end{equation}
Hence, for $C_1$ and $C_0$ independent on $D$, we have for $D$ large enough and for any $l \in \mathbb{N}$,

\begin{equation}
    \left\{
    \begin{aligned}
    &|S_l \cdot F(S_l)| \leq \frac{C_0 C_1^{l+1}}{D^{l+2}},\\
    &|S_l - S_0| \leq \varepsilon |S_0|.
    \end{aligned}
    \right.
\end{equation}
We deduce that $S_l$ is a Cauchy sequence, converging to $S_\infty$, satisfying

\begin{equation}
    \left\{
    \begin{aligned}
        &S_\infty^2 =0,\\
        &S_\infty \cdot F(S_\infty) = 0.
    \end{aligned}
    \right.
\end{equation}
Now, $Re(x_i)$ and $\im(x_i)$ are defined, but $\dot x_i(0) $ are not. We will estimate them now. We will use that 

\begin{equation}
    w_j = \frac{(S_0)_j(0) \times (\bar \s_0)_j(0) }{|(S_0)_j(0)^2|}\cdot \iu \m_0,
\end{equation}
and compare it with 

\begin{equation}
    \dot x_j(0) = \frac{(S_\infty)_j(0) \times (\bar S_\infty)_j(0) }{|(S_\infty)_j(0)^2|}\cdot \left( \iu \m_0 - \sum_{k\neq j} \frac{(S_\infty)_k(0)}{x_j(0)-x_k(0)} + \sum_{k\neq j} \frac{(\bar \s_\infty)_k(0)}{x_j(0)-\bar x_k(0)}  \right).
\end{equation}
Since $w_j\neq \pm 1$, then $\gamma_j \neq 0$, so $\im(x_j)$ is bounded. Then, 

\begin{equation}
   \left|  \frac{(S_0)_j \times (\bar S_0)_j }{|(S_0)_j^2|} - \frac{(S_\infty)_j \times (\bar S_\infty)_j }{|(S_\infty)_j^2|}  \right| \leq \left| (S_0)-(S_\infty) \right| \cdot ||\nabla_s h(s)||_\infty,
\end{equation}
with 

\begin{equation}
    h: \left\{
    \begin{aligned}
        &s\in \mathbb{C}^3 \to \mathbb{C}^3,\\
        &s \mapsto \frac{s\times s}{|s|^2}.
    \end{aligned}
    \right.
\end{equation}
Hence, 

\begin{multline}
   \left|  \frac{(S_0)_j \times (\bar S_0)_j }{|(S_0)_j^2|} - \frac{(S_\infty)_j \times (\bar S_\infty)_j }{|(S_\infty)_j^2|}  \right| \leq \left| (S_0)-(S_\infty) \right| \cdot \left( \frac{2 |s|^3 + 2 |s|^3}{|s|^4} \right)_\infty\\
   \leq \varepsilon |S_0| \sup_{[S_0,S_\infty]}\frac{4}{|s|} \leq 4 \frac{\varepsilon}{1-\varepsilon}.
\end{multline}
Finally, we make an estimation for the second term appearing in the difference between $\dot x_j(0)$ and $w_j$. We find an upper bound for the two terms involved in the inner product. 

\begin{equation}
    \left|\frac{(S_\infty)_j \times (\bar S_\infty)_j }{|(S_\infty)_j^2|} \right| \leq 1,
\end{equation}
and with $I = \max |\im(x_i)|$,

\begin{equation}
    \left| \sum_{j\neq k} \frac{(S_\infty)_k}{x_j(0)-x_k(0)} \right| \leq \frac{N (1+\varepsilon) |S_0|}{\nu D - 2 I}, 
\end{equation}

\begin{equation}
    \left| \sum_{j\neq k} \frac{(\bar S_\infty)_k}{x_j(0)-\bar x_k(0)} \right| \leq \frac{N (1+\varepsilon) |S_0|}{\nu D}.
\end{equation}
Hence, the distance between the two speeds satisfy

\begin{equation}
    |\dot x_j(0) - w_j| \leq 8 \varepsilon + \frac{N(1+\varepsilon)|S_0|}{\nu D -2 I} + \frac{N (1+\varepsilon) |S_0|}{\nu D}.
\end{equation}
In particular, for $D$ big enough, $|\dot x_j(0) - w_j| \leq 12 \varepsilon$, which means that we obtain an initial condition satisfying all the initial conditions, and 

\begin{equation}
    \min |x_i(0) - x_j(0)| \geq \nu /2.
\end{equation}

\subsection{Asymptotic spectrum of the created solution}

To summarize, for a given spectrum $(w_1,\dots,w_N)$, we can create an initial condition for any $\varepsilon$ with spins $(s_1,\dots,s_N)$ and poles $(x_1,\dots,x_N)$ such that for all $j$,
\begin{equation}
\left\{
    \begin{aligned}
        &|\s_j(0)| \leq 10,~\forall j,\\
        &|\dot x_j(0) - w_j| \in [-12\varepsilon, 12\varepsilon],\\
        &\min |x_j(0)-x_k(0)| = \nu D.
    \end{aligned}
    \right.
\end{equation}

Hence, for $D$ big enough, the instant scattering condition required for Theorem \ref{thm:localCondition} to apply, is automatically satisfied. Using the triangular inequality

\begin{equation}
    |v_j - w_j| \leq |v_j - \dot x_j(0)| + |w_j - \dot x_j(0)|,
\end{equation}

we obtain a solution with a spectrum as close as we want to $(w_j)$.

    \section{Systematic scattering, trivial scattering map}\label{sec:sys}
	\subsection{Guaranteed scattering behavior for large times}
	
	In this section, we assume that $v_i\neq v_j$, where $v_i$ are the eigenvalue of $T_{U_0}$. We show that the solution of the half wave map scatters at $t = + \infty$, i.e. that after a time $t_0$, the solution satisfies Proposition \ref{scattering}. To show this, we show that there exists $t_0>0$ such that 
	
	\begin{equation}
		\alpha(x,s)(t) = D - \frac{16 N S}{\nu} >0,
	\end{equation}
	for $t\geq t_0$. It is a consequence of the two following lemmas.
	
	\begin{lemma}\label{lemma:poles}
		There exists $t_0 \geq 0$, such that $t \geq t_0$ implies
		\begin{equation}
			|x_i(t) - x_j(t)| \geq \frac{t \cdot \min_{i\neq j} |v_i-v_j| }{2}.
		\end{equation}
	\end{lemma}
	
	\begin{lemma}\label{lemma:spins}
		There exists $C>0$ such that for any $t \geq 0$, for any $j$,
		\begin{equation}
			|\s_j(t)| \leq C.
		\end{equation}
	\end{lemma}
	
	We first start with the proof of Lemma \ref{lemma:spins}.
	
	\begin{proof}
	
	In this section, we will provide an asymptotic formula assuming that $v_i \neq v_j$ for $i\neq j$, and deduce the boundedness of the spins. We use the explicit formula for (HWM) provided by \cite{ohlmann2024halfwavemapsexplicitformulas}, which reads
	
	\begin{equation}\label{expl}
		\Pi_- \m(t,x) = - T^T \mathcal{E}_0 \mathcal{H} [X_0+tL-x I_N]^{-1} \mathcal{F}_0 T,
	\end{equation}
    where $\mathcal{E}_0$ and $\mathcal{F}_0$ are the $2N\times 2N$ half-spin matrices at $t=0$, $T$ is a column of $N$ $2 \times 2$ identity matrices, and $\mathcal{H}$ is a $2N\times 2N$ matrix formed by a diagonal of $\begin{pmatrix}
        1 & 1 \\
        1 & 1
    \end{pmatrix}$ matrices. Finally, for a matrix $U \in \mathcal{C}^{N\times N}$, $[U]$ is the doubled matrix composed of $N\times N$ blocks of size $2\times 2$, where the block at position $(i,j)$ is given by $(U)_{i,j} I_{2}$. 
    
    Since $L$ has $N$ distinct eigenvalues $v_1,\dots,v_N$ by assumption, we can write $L=P J P^{-1}$ with $J=\operatorname{Diag}(v_1,\dots,v_N)$. Hence, \eqref{expl} can be rewritten as 
    \[
    \Pi_- \m(t,x)  = - T^T \mathcal{E}_0 \mathcal{H} [P] [P^{-1}X_0 P + t J - xI_2]^{-1} [P]^{-1} \mathcal{F}_0 T.
    \]
    We define the decomposition of $[P]^{-1} \mathcal{F}_0 T$ into 'eigenvectors' of $[J]$ in the following way. 
    \[
    \mathcal{K} = [P]^{-1}\mathcal{F}_0 T = \begin{pmatrix}
        K_1 \\
        K_2 \\
        \vdots \\
        K_N
    \end{pmatrix},\quad \mathcal{K}_j = \begin{pmatrix}
        0 \\
        \vdots \\
        0 \\ 
        K_j \\
        0 \\
        \vdots \\
        0
    \end{pmatrix}, 
    \quad \mathcal{K} = \sum_{j=1}^N \mathcal{K}_j.
    \]
    The first step is to establish the two following estimates. With $G=P^{-1}X_0 P$,

    \begin{equation}\label{estest}
    \left\{
        \begin{aligned}
           &[G+ tJ - (t v_i + \sqrt{t})I_N]^{-1} \mathcal{K}_j = \frac{\mathcal{K}_j}{t (v_j-v_i)} + O \left( \frac{1}{t^{3/2}} \right),~j\neq i,\quad (i)\\
           &[G+tJ-(t v_i + \sqrt{t}) I_N]^{-1} \mathcal{K}_i = - \frac{\mathcal{K}_i}{\sqrt{t}} + O \left( \frac{1}{t} \right) . \quad (ii)
        \end{aligned}
    \right.
    \end{equation}
    We start with the proof of $\eqref{estest}-(i)$, and write

    \[
    [G+tJ - (v_i t + \sqrt{t}) I_N] \mathcal{K}_j = [G] \mathcal{K}_j + (t v_j - t v_i - \sqrt{t})  \mathcal{K}_j,
    \]
    so we obtain
    \[
    [G + t J - (v_i t + \sqrt{t})]^{-1} \mathcal{K}_j = \frac{\mathcal{K}_j}{t v_j-t v_i - \sqrt{t}} - \frac{[G + t J - (v_i t + \sqrt{t})]^{-1} [G] \mathcal{K}_j}{t v_j-t v_i - \sqrt{t}}
    \]
    Now, $||G + t T_{U_0} - x||^{-1} \leq \frac{C}{\sqrt{t}}$, since in some basis, we have
		\begin{equation}
			 t T_{U_0} -x  = \begin{pmatrix}
				t v_1 - t v_i -\sqrt{t} & 0 & 0 & \dots \\
				0 & t v_2 - t v_i - \sqrt{t} & 0 & \dots\\
				\dots & \dots & \dots & \dots \\
				0 & 0 & 0 & t v_N - t v_i -\sqrt{t},
			\end{pmatrix}
		\end{equation}
	and $G$ is bounded. This means that we have by Gershgorin's circle theorem 
    \[
    \left| \frac{[G + t J - (v_i t + \sqrt{t})]^{-1} [G] \mathcal{K}_j}{t v_j-t v_i - \sqrt{t}} \right| \leq \frac{C'}{t^{3/2}},
    \]
    so 
    \[
    [G+ tJ - (t v_i + \sqrt{t})I_N]^{-1} \mathcal{K}_j = \frac{\mathcal{K}_j}{t (v_j-v_i)} + O \left( \frac{1}{t^{3/2}} \right).
    \]
    We go on with the proof of $\eqref{estest}-(ii)$. We have 
    \[
    [G+t J - (tv_i + \sqrt{t}) I_N] \mathcal{K}_i = [G] \mathcal{K}_i - \sqrt{t} \mathcal{K}_i.
    \]
    Consequently, we have 
    \[
    [G+t J - (tv_i + \sqrt{t}) I_N]^{-1} \mathcal{K}_i = - \frac{\mathcal{K}_i}{\sqrt{t}} + \frac{[G+t J - (tv_i + \sqrt{t}) I_N]^{-1} [G] \mathcal{K}_i}{\sqrt{t}}.
    \]
    We hence obtain using the same argument
    \[
    [G+t J - (tv_i + \sqrt{t}) I_N]^{-1} \mathcal{K}_i = - \frac{\mathcal{K}_i}{\sqrt{t}} + O \left( \frac{1}{t} \right).
    \]
	
	Note that $x = t v_i + \sqrt{t}$ has been choosen to obtain a moving Cauchy determinant, since $x$ is close to the pole in comparison with the others, but far enough to make an estimate. Estimate \eqref{estest} implies
    \[
    \Pi_- \m(t,x) = - T^T \mathcal{E}_0 \mathcal{H} [P] \sum_{j=1}^N [G + t J - x I_N]^{-1} \mathcal{K}_j = \frac{1}{\sqrt{t}} T^T \mathcal{E}_0 \mathcal{H} [P] \mathcal{K}_i + O \left( \frac{1}{t} \right),
    \]
    and we have 
    \[
    \Pi_- \m(t,x) = \sum_{j=1}^N \frac{A_j(t)}{tv_i + \sqrt{t} - x_j(t)}.
    \]
    For $t$ large enough, $|x_i(t)-x_j(t)| \geq t |v_i-v_j| $ for $j\neq i$.
    Another way to write this is
	
	\begin{equation}
		\begin{bmatrix}
			\frac{1}{t v_1 + \sqrt{t} - x_1(t)}  & \frac{1}{t v_1 + \sqrt{t} - x_2(t)} & \dots &  \frac{1}{t v_1 + \sqrt{t} - x_N(t)} \\
			\dots & \dots & \dots &\dots \\
			\dots & \dots & \dots &\dots \\
			\frac{1}{t v_N + \sqrt{t} - x_1(t)}  & \frac{1}{t v_N + \sqrt{t} - x_2(t)} & \dots &  \frac{1}{t v_N + \sqrt{t} - x_N(t)} \\
		\end{bmatrix}
		\begin{bmatrix}
			A_1(t) \\
			A_2(t) \\
			\vdots \\
			A_N(t)
		\end{bmatrix} = \begin{bmatrix}
			L_1(t) \\
			L_2(t) \\
			\vdots \\
			L_N(t)
		\end{bmatrix},
	\end{equation}
	where $L_j(t) =  \frac{1}{\sqrt{t}} T^T \mathcal{E}_0 \mathcal{H} [P] \mathcal{K}_i + O \left( \frac{1}{t} \right)$. The left hand side is a Cauchy matrix with $a_i = t v_i + \sqrt{t}$ and $b_j = -x_j(t)$. The distances between the $a_i$ are at least $|v_i - v_j| \cdot t$, the distances between $b_i$ are controlled by $|v_i - v_j| \cdot t$, and the distances between $a_i$ and $b_j$ are controlled by $|v_i - v_j|\cdot t$, except for $i = j$ where it is $\OO(\sqrt{t})$.
	
	Recall that a Cauchy matrix is a matrix of the form 
	$$C_n = \begin{bmatrix}
		\dfrac 1 {a_1 + b_1} & \dfrac 1 {a_1 + b_2} & \cdots & \dfrac 1 {a_1 + b_N} \\ \dfrac 1 {a_2 + b_1} & \dfrac 1 {a_2 + b_2} & \cdots & \dfrac 1 {a_2 + b_N} \\ \vdots & \vdots & \ddots & \vdots \\ \dfrac 1 {a_n + b_1} & \dfrac 1 {a_N + b_2} & \cdots & \dfrac 1 {a_N + b_N} \\ \end{bmatrix},$$
	and that its inverse is given by 
	
	\begin{equation}
		\begin{bmatrix} b_{ij} \end{bmatrix} = \begin{bmatrix} \dfrac { \prod_{k = 1}^N \left( a_j + b_k\right) \left( a_k + b_i \right) } {\left( a_j + b_i \right) \left( \prod_{\substack {1 \mathop \le k \mathop \le N \\ k \mathop \ne j} } \left( a_j - a_k\right) \right) \left( {\prod_{\substack {1 \mathop \le k \mathop \le N \\ k \mathop \ne i} } \left( b_i - b_k\right) } \right) } \end{bmatrix}.
	\end{equation}
    We provide an estimation of $b_{ij}$. We have for the numerator
    \[
    \begin{aligned}
        & \prod_{k=1}^N (t v_j + \sqrt{t} - x_k) (t v_k + \sqrt{t} - x_i) \\
        &= \prod_{k\neq j }^N (t v_j + \sqrt{t} - x_k) \prod_{k\neq i}^N (t v_k + \sqrt{t} - x_i) \cdot (t v_j + \sqrt{t} - x_j) (t v_i + \sqrt{t} - x_i) \\
        &\sim_t t^{2N-2} \cdot t \cdot \prod_{k\neq j} (v_j-v_k) \prod_{k\neq i} (v_k-v_i) \sim C_1 t^{2N-1},
    \end{aligned}
    \]
    and for the denominator
    \[
    a_j+b_i = t v_j + \sqrt{t} - x_i \sim_t \left\{
    \begin{aligned}
        &t (v_j-v_i),~j\neq i,\\
        &\sqrt{t},~i=j,
    \end{aligned}
    \right.
    \]
    and 
    \[
    \prod_{k\neq j} (a_j-a_k) \prod_{k\neq i} (b_i-b_k) \sim C_2 t^{2N-2}.
    \]
    Hence, we have for the coefficients $b_{ij}$ of the Cauchy matrix
    \[
    b_{ij} \sim \left\{
    \begin{aligned}
        & \OO(1),~i\neq j,\\
        & \OO(\sqrt{t}),~ i =j.
    \end{aligned}
    \right.
    \]
	Subsequently, the operator norm of $C_N^{-1}$ satisfies
	
	\begin{equation}
		||C_N|| \leq \max_{i=1}^N \sum_{j=1}^N |b_{i,j}| \leq C \sqrt{t},
	\end{equation}
	and
	\begin{equation}
		\begin{bmatrix}
			| A_1(t) | \\
			| A_2(t) | \\
			\vdots \\
			| A_N(t) |
		\end{bmatrix} \leq ||C_N^{-1}|| \begin{bmatrix}
			| L_1(t) | \\
			| L_2(t) | \\
			\vdots \\
			| L_N(t) |
		\end{bmatrix} \leq \OO(1),
	\end{equation}
    so $A_j(t)$ and $s_j(t)$ are bounded.
    
	\end{proof}
	 We go on with the proof of Lemma \ref{lemma:poles}.
	
	\begin{proof}
        Since the poles are given by the eigenvalues of $X_0 + t L$, which are also the eigenvalues of $G+t J$ for a bounded matrix $G$, we again obtain since the diagonal coefficients are equal to $v_i t$ and the off-diagonal coefficients are bounded, that 
        \[
        x_i(t) = tv_i + \OO(1).
        \]
    
	\end{proof}

	\subsection{The scattering map is the identity}
	
	We consider $\m$ a rational solution of the Cauchy problem
    \begin{equation}\label{sec7:HWM}
    \left\{
    \begin{aligned}
        &\partial_t \m(t,x) = \m(t,x) \times |\nabla| \m(t,x),\\
        &m(0,x) = \m_0.
    \end{aligned}
    \right.
    \end{equation}
    Equivalently, we consider $\m$ the solution of the $2t\times 2$ version of the same problem, with $\m_0 = \m_0 \cdot \sigma$,
    \begin{equation}\label{sec7:HWMmat}
    \left\{
    \begin{aligned}
        &\partial_t \m(t,x) = - \frac{\iu}{2} \left[ \m(t,x),|\nabla| \m(t,x) \right],\\
        &M(0,x) = \m_0.
    \end{aligned}
    \right.
    \end{equation}
    We denote 
    \begin{equation}\label{sec7:defM}
    \begin{aligned}
        &\m(t,x) = m_\infty + \sum_{j=1}^\infty \frac{\s_j(t)}{x-x_j(t)} + \sum_{j=1}^\infty \frac{\bar \s_j(t)}{x-\bar x_j(t)},\\
        &\m(t,x) = M_\infty + \sum_{j=1}^\infty \frac{A_j(t)}{x-x_j(t)} + \sum_{j=1}^\infty \frac{A_j^*(t)}{x-\bar x_j(t)}.
    \end{aligned}
    \end{equation}
    
    We assume that the speeds are distinct, $v_i\neq v_j$. We have shown the existence of rational functions 
    \begin{equation}\label{sec7:defAsym}
    \begin{aligned}
    &f_+(x,t) = \sum_{j=1}^N \frac{r_j^+}{x-v_j-a_j^+},
    ~f_-(x,t) = \sum_{j=1}^N \frac{r_j^-}{x-v_j-a_j^-}, \\
    &F_+(x,t) = \sum_{j=1}^N \frac{R_j^+}{x-v_j-a_j^+},
    ~F_-(x,t) = \sum_{j=1}^N \frac{R_j^-}{x-v_j-a_j^-},
    \end{aligned}
    \end{equation}
    such that the solution scatters in Sobolev norm, i.e, 
    \begin{equation}\label{sec7:scatt}
    \begin{aligned}
    &\sobonorm{\Pi_- \m(t,x)-f_+(x,t)}{1/2} \xrightarrow[t\rightarrow +\infty]{} 0, 
    ~\sobonorm{\Pi_- \m(t,x)-f_-(x,t)}{1/2} \xrightarrow[t\rightarrow -\infty]{} 0,\\
    &\sobonorm{\Pi_- \m(t,x)-F_+(x,t)}{1/2} \xrightarrow[t\rightarrow +\infty]{} 0, 
    ~\sobonorm{\Pi_- \m(t,x)-F_-(x,t)}{1/2} \xrightarrow[t\rightarrow -\infty]{} 0.
    \end{aligned}
    \end{equation}
    We here study the scattering map $\phi$ or equivalently $\Phi$
    \[
    \phi:~f_- \mapsto f_+,\quad \Phi:~F_- \mapsto F_+,
    \]
    which is not yet properly defined. We now state the main theorem of this section, which is that the scattering map is the identity provided that the function in the middle (for finite times) corresponds to a non-degenerate spectrum. 

    \begin{theorem}
        Let $\m$ be a solution of the (HWM) Cauchy problem \eqref{sec7:HWM} and $\m$ the corresponding $2\times 2$ matrix version, the solution of the corresponding (HWM) Cauchy problem as defined in \eqref{sec7:HWMmat}. 
        We consider the expressions of $\m$ and $\m$ as given by \eqref{sec7:defM} and define the spins $s_j(t)$ and $A_j(t)$ and the poles $x_j(t)$ accordingly.
        We assume $v_j \neq v_k$ for $j\neq k$.

        By Proposition \ref{sec6:sysScatt}, there exists $f_+,f_-:\mathbb{R}^2 \to \mathbb{C}^3$ and $F_+,F_-:\mathbb{R}^2 \to \mathbb{C}^{2\times 2}$, with expressions given by \eqref{sec7:defAsym} and such that the solutions scatter as in \eqref{sec7:scatt}.

        Then, we have, 
        \[
        \boxed{
        f_+(t,x) = f_-(t,x),\quad F_+(t,x) = F_-(t,x),~\forall t,x\in \mathbb{R}.
        }
        \]

        Moreover an explicit asymptotic expression of the spins can be computed. Let's denote $E_1,\dots,E_N$ and $F_1,\dots,F_N$ the half spins corresponding to $\m$ at $t=0$, in the sense $E_j H F_j = A_j(0)$, and $P$ the $N \times N$ matrix such that $L(0)$ as defined in \cite{ohlmann2024halfwavemapsexplicitformulas} satisfies
        \[
        L(0)=P J P^{-1},~ J = \begin{pmatrix}
            v_1 & 0 & \dots & 0 \\
            0 & v_2 & \dots & 0 \\
            \vdots & 0 & \ddots & \vdots \\
            0 & \dots & \dots & v_N
        \end{pmatrix}.
        \]
        Then, with $[P^{-1}]$ being the doubled matrix of $P^{-1}$ and 
        \[
        \begin{pmatrix}
            K_1\\
            K_2\\
            \vdots \\
            K_N
        \end{pmatrix} = [P]^{-1} \begin{pmatrix}
            F_1 \\
            F_2 \\
            \vdots \\
            F_N
        \end{pmatrix},
        \]
        the asymptotic spin $R_j$ satisfies
        \[
        \boxed{
        R_j^+= R_j^- = - (E_1,\dots,E_N) \mathcal{H} [P] (0,\dots,0,K_j,0,\dots,0)^T,}
        \]
        or equivalently
        \[
        \boxed{
        R_j^+ = R_j^- = - \begin{pmatrix}
            E_1 \\
            E_2 \\
            \vdots \\
            E_N
        \end{pmatrix}^T \mathcal{H} [P] \scalebox{0.7}{$
\begin{pmatrix}
0 & & & & \\
& \ddots & & & \\
& & 0 & & \\
& & & 1 & \\
& & & & 0 \\
& & & & & \ddots \\
& & & & & & 0
\end{pmatrix}
$} [P]^{-1} \begin{pmatrix}
    F_1 \\
    F_2 \\
    \vdots \\
    F_N
\end{pmatrix}.
        }
        \]
        
    \end{theorem}

Note that this theorem contains two parts, namely $a_j^+ = a_j^-,~ \forall j$, and $r_j^+ = r_j^-,~ \forall j$.
We first show $r_j^+ = r_j^-$, using that $x_j(t) = v_j t + \OO(1)$.
\begin{proof}
    We will use the explicit formula 
    \[
        \Pi_- \m(t,x) = T^T \mathcal{E}(0) \mathcal{H} \left[ X + t L(0) - xI \right]^{-1} \mathcal{F}(0) T. 
    \]
    With $L(0) = P J P^{-1}$, where $J = diag(v_1,\dots,v_N)$,
    \[
    \left[ X + t P J P^{-1} - x I \right]^{-1} = [P] [P^{-1} X P + t J - x I]^{-1} [P^{-1}], 
    \]
    so
    \[
    \Pi_- \m(t,x) = T^T \mathcal{E}(0) \mathcal{H} [P] [G_0 + t J - x I]^{-1} [P]^{-1} \mathcal{F}(0) T.
    \]
    We now consider 
    \[
    [G+ t J - x I]^{-1} [P]^{-1} \begin{pmatrix}
        F_1 \\
        F_2 \\
        \vdots \\
        F_N
    \end{pmatrix} = [G+ t J - x I]^{-1} \begin{pmatrix}
        K_1 \\
        K_2 \\
        \vdots \\
        K_N
    \end{pmatrix}. 
    \]
    For $x=t v_k + \sqrt{t}$,
    \[
    [G+tJ-xI] \begin{pmatrix}
        K_1 \\
        K_2 \\
        \vdots \\
        K_N
    \end{pmatrix} = [G] \begin{pmatrix}
        K_1 \\
        K_2 \\
        \vdots \\
        K_N
    \end{pmatrix} + t \begin{pmatrix}
        (v_1-v_k) K_1 \\
        (v_2-v_k) K_2 \\
        \vdots \\
        (v_N-v_k) K_N
    \end{pmatrix} - \sqrt{t} \begin{pmatrix}
        K_1 \\
        K_2 \\
        \vdots \\
        K_N
    \end{pmatrix}.
    \]
    With 
    \[
    \mathcal{K}= \begin{pmatrix}
        K_1 \\
        K_2 \\
        \vdots \\
        K_N
    \end{pmatrix},\quad \mathcal{K}_j = \begin{pmatrix}
        0 \\
        \vdots \\
        0 \\
        K_j \\
        0 \\
        \vdots \\
        0
    \end{pmatrix},
    \]
    we have 
    \[
    \mathcal{K} = [G+tJ-xI]^{-1} [G] \mathcal{K} +  [G+tJ-xI]^{-1} t \cdot diag(v_j-v_k) \cdot \mathcal{K}  -\sqrt{t} [G+tJ-xI]^{-1} \mathcal{K},
    \]
    and since $[G+tJ-xI]^{-1} t \cdot diag(v_j-v_k) \cdot \mathcal{K}_k = 0$,
    \[
    \mathcal{K}_k = [G+tJ-xI]^{-1}[G]\mathcal{K}_k - \sqrt{t} [G+tJ-xI]^{-1}\mathcal{K}_k,
    \]
    or in other words 
    \[
    [G+tJ-xI]^{-1}\mathcal{K}_k = \frac{[G+tJ-xI]^{-1}[G]\mathcal{K}_k}{\sqrt{t}} - \frac{\mathcal{K}_k}{\sqrt{t}}.
    \]
    $G+tJ-xI$ is diagonal dominant, with diagonal coefficients satisfying $d_{k,k} = \sqrt{t} + \OO(1)$, $d_{j,j} = (v_j-v_k) t + \OO(\sqrt{t})$ for $j\neq k$, and bounded non-diagonal coefficients. Consequently, Gershgorin's circle theorem provides that the eigenvalues $\lambda_j(t)$ of $G+tJ-xI$ satisfy $\lambda_k(t) = \sqrt{t} + \OO(1)$ and $\lambda_j(t)=(v_j-v_k) t + \OO(\sqrt{t})$ for $j \neq k$. The eigenvalues $\mu_j$ of $(G+tJ-xI)^{-1}$ satisfy 
    \[
    \mu_k(t) = \frac{1}{\sqrt{t} + \OO(1)},~\mu_j(t) = \frac{1}{t (v_j-v_k)+\OO(1)}.
    \]
    Hence, the operator norm of $[G+tJ-xI]^{-1}$ satisfy 
    \[
    \left|\left|\left| [G+tJ-xI]^{-1} \right|\right|\right| \leq \frac{2}{\sqrt{t}}.
    \]
    implying in particular 
    \[
    \left|\left| \frac{[G+tJ-xI]^{-1}[G] \mathcal{K}_k}{\sqrt{t}} \right|\right| = O\left( \frac{1}{t} \right),
    \]
    and so
    \[
    [G+tJ-xI]^{-1} \mathcal{K}_k = - \frac{\mathcal{K}_k}{\sqrt{t}} + O \left( \frac{1}{t}\right).
    \]
    Now, for $j\neq k$,
    \[
    \mathcal{K}_j = [G+tJ-xI]^{-1} [G] \mathcal{K}_j + t (v_j-v_k) [G+tJ-xI]^{-1} \mathcal{K}_j - \sqrt{t} [G+tJ-xI]^{-1} \mathcal{K}_j,
    \]
    so
    \[
    [G+tJ-xI]^{-1} \mathcal{K}_j = \frac{1}{t (v_j-v_k) - \sqrt{t}} \left( \mathcal{K}_j - [G+tJ-xI]^{-1} [G] \mathcal{K}_j \right) = O \left( \frac{1}{t} \right),
    \]
    using a similar argument. Hence,

    \[
    [G+tJ-xI]^{-1} \mathcal{K} = \sum_{j=1}^N [G+tJ-xI]^{-1} \mathcal{K}_j = - \frac{\mathcal{K}_k}{\sqrt{t}} + O \left( \frac{1}{t} \right).
    \]
    Hence, 
    \[
    \Pi_- \m(t,x) = - T \mathcal{E}(0) \mathcal{H} [P] [G+tJ-xI]^{-1} \mathcal{K} = \frac{1}{\sqrt{t}} (E_1,\dots,E_N) [P] \mathcal{K}_k + O \left( \frac{1}{t}\right).
    \]
    Now, using again that $x=v_k t + \sqrt{t}$ and the expression of $\Pi_- M$, we get
    \[
    \Pi_- \m(t,x) = \sum_{j=1}^N \frac{A_j(t)}{v_k t + \sqrt{t} - x_k(t)}.
    \]
    Since  $x_j(t) = v_j t + \OO(1)$, we get 
    \[
    \Pi_- \m(t,x) = \frac{A_j(t)}{\sqrt{t}} + O \left( \frac{1}{t} \right).
    \]
    In particular, we obtain 
    \[
    \boxed{
    A_j(t) = (E_1,\dots,E_N) \mathcal{H} [P] \mathcal{K}_j + O\left( \frac{1}{t} \right),
    }
    \]
    so 
    $A_j(t) \rightarrow (E_1,\dots,E_N) \mathcal{H} [P] \mathcal{K}_j$ as $t \rightarrow \infty$, which implies that 
    \[
    \boxed{
    R_j^+ = (E_1,\dots,E_N) \mathcal{H} [P] \mathcal{K}_j.
    }
    \]
    We now consider $t \rightarrow -\infty$, and $x=v_i t + \sqrt{|t|}$. With the same notations, we get 
    \[
    \left\{
    \begin{aligned}
        &[G+t J - xI]^{-1} \mathcal{K}_j = \frac{[G+t J - xI]^{-1}[G]\mathcal{K}_j-\mathcal{K}_j}{\sqrt{|t|}},\\
        &[G+t J - xI]^{-1} \mathcal{K}_k = \frac{K_k-[G+t J -x I]^{-1}[G]K_k}{t (v_k-v_j) - \sqrt{|t|}}.
    \end{aligned}
    \right.
    \]
    The same argument gives 
    \[
    \left\{
    \begin{aligned}
        &[G+t J - xI]^{-1} \mathcal{K}_j = \frac{-\mathcal{K}_j}{\sqrt{|t|}}+O\left( \frac{1}{t}\right),\\
        &[G+t J - xI]^{-1} \mathcal{K}_k = O\left( \frac{1}{t} \right),
    \end{aligned}
    \right.
    \]
    and using again $\Pi_- \m(t,x) = - T \mathcal{E}_0 \mathcal{H} [P] [G+tJ-xI]^{-1} [P]^{-1}\mathcal{F}_0 T^T$ at the point $x = v_j t + \sqrt{|t|}$ as $t \rightarrow \infty$ we get 
    \[
    \Pi_- \m(t,x) = - T \mathcal{E}_0 \mathcal{H} [P] \frac{- \mathcal{K}_j}{\sqrt{|t|}} + O\left( \frac{1}{t} \right).
    \]
    On the other hand, we also use the rational expression of $\m$ for this $x$ to get
    \[
    \Pi_- \m(t,x) = \sum_{j=1}^N \frac{A_j(t)}{x-x_j(t)} = \frac{A_j(t)}{\sqrt{|t|}} + O \left( \frac{1}{t} \right).
    \]
    In particular, 
    \[
    A_j(t) \xrightarrow[t \rightarrow -\infty]{} T \mathcal{E}_0 \mathcal{H} [P] \mathcal{K}_j,
    \]
    which implies that 
    \[
    \boxed{
    R_j^+ = R_j^-.
    }
    \]
    We now go on and show that $a_j^+ = a_j^-$. We consider $x = t v_j$ and write
    \[
    [G+ t L - x I] \mathcal{K}_k = [G] \mathcal{K}_k + t v_k \mathcal{K}_k - t v_j \mathcal{K}_k.
    \]
    Since $[G]\mathcal{K}_k = \sum_{i=1}^N (G)_{i,k} \mathcal{K}_i $,
    \[
    [G+t L -xI] \mathcal{K}_k = \left( t (v_k-v_j) + G_{k,k} \right) \mathcal{K}_k + \sum_{i \neq k} (G)_{i,k} \mathcal{K}_i.
    \]
    We assume for now that $(G)_{j,j}\neq 0$. For $j=k$, we obtain 
    \[
    [G+tL-xI]^{-1} \mathcal{K}_j = \frac{\mathcal{K}_j - \sum_{i\neq j} (G)_{i,j} [G+t L - xI]^{-1}\mathcal{K}_i }{G_{j,j}},
    \]
    and for $k\neq j$,
    \[
    \begin{aligned}
    [G+t L - x I]^{-1} \mathcal{K}_k &= \frac{\mathcal{K}_k - [G+t L - x I]^{-1} [G] \mathcal{K}_k}{t (v_k-v_j)}\\
    &= \frac{\mathcal{K}_k -  \sum_{i\neq j}^N (G)_{i,k} [G+t L - x I]^{-1} \mathcal{K}_i - (G)_{j,k} [G+t L - x I]^{-1} \mathcal{K}_j}{t (v_k-v_j)}.
    \end{aligned}
    \]
    In other words, with $g_j(t) = [G+tL-xI]^{-1} \mathcal{K}_j$, 
    \[
    \left\{
    \begin{aligned}
        &g_j(t) + \sum_{i\neq j} \frac{(G)_{i,j}}{(G)_{j,j}} g_i(t) = \frac{\mathcal{K}_j}{(G)_{j,j}},\\
        &g_k(t) + \sum_{i=1}^N \frac{(G)_{i,k}}{t(v_k-v_i)} g_i(t) = \frac{\mathcal{K}_k}{t(v_k-v_j)}.
    \end{aligned}
    \right.
    \]
    Considering the $N-1$ lines for $j\neq k$, we get 
    \[
    \frac{1}{t} V g_j(t) + I_{N-1} \tilde{Y}(t) + \frac{1}{t} M \tilde{Y}(t) = \frac{1}{t} \tilde{K},
    \]
    with 
    \[
    (V)_k = \frac{(G)_{j,k}}{v_k-v_j},\quad (M)_{i,k} = \frac{(G)_{k,i}}{(v_i-v_k)},\quad (\tilde{K})_k = \frac{\mathcal{K}_k}{(v_k-v_j)},\quad (\tilde{Y})_k(t) = g_k(t),~ k\neq j,
    \]
    which can be rewritten as $\left( I + \frac{1}{t} M \right) \tilde{Y} = \frac{1}{t} \tilde{K} - \frac{1}{t} V g_j(t)$. Since, the operator norm of $\m$ is a constant not depending on $t$, we have for $|t|$ large enough that $\left(I-\frac{1}{t} M \right)$ is invertible and 
    \[
    \left|\left|\left| \left( I + \frac{1}{t} M \right)^{-1} - I \right|\right|\right| \leq \frac{2 |||M|||}{t} \leq \frac{C}{t}.
    \]
    We can finally conclude that $g_k(t)=\oo(g_j(t))$. Indeed,
    \[
    \begin{aligned}
    \left|\left| \tilde{Y} - \left( \frac{1}{t} \tilde{K} - \frac{1}{t} V g_j(t) \right) \right|\right| &= \left|\left| \left( \left( I + \frac{1}{t} M \right)^{-1} - I \right) \left( \frac{1}{t} \tilde{K} - \frac{1}{t} V g_j(t) \right) \right|\right| \\
    & \leq \frac{C}{t} \left( \left|\left| \frac{1}{t}\tilde{K} \right|\right| + 
    \left|\left| \frac{1}{t} V g_j(t) \right|\right| \right),
    \end{aligned}
    \]
    which means
    \[
    \tilde{Y} = O \left( \frac{1}{t} \right) - \frac{1}{t} V g_j(t) + O\left( \frac{1}{t^2}\right) + O\left(\frac{1}{t^2} V g_j(t) \right) = O \left( \frac{1}{t} \right) + O \left( \frac{g_j(t)}{t}\right).
    \]
    Now, we look at the line given by the index $j$, using $g_k(t)= O\left( \frac{1}{t} \right) + O \left( \frac{g_j(t)}{t} \right)$, and obtain that the dominant term is given by $g_j(t)$, indeed
    \[
    g_j(t) + \sum_{i\neq j} \frac{(G)_{i,j}}{(G)_{j,j}} \left[ O\left( \frac{1}{t} \right) + O\left( \frac{g_j(t)}{t} \right) \right] = \frac{\mathcal{K}_j}{(G)_{j,j}},
    \]
    so 
    \[
    g_j(t) = \frac{\mathcal{K}_j}{(G)_{j,j}} + O \left( \frac{1}{t} \right).
    \]
    The explicit formula now gives
    \[
    \Pi_- \m(t,x) = \sum_{k=1}^N - (E_1,\dots,E_N) \mathcal{H} [P] g_k(t) = - \frac{1}{(G)_{j,j}} (E_1,\dots,E_N) \mathcal{H} [P] \mathcal{K}_j + O \left( \frac{1}{t} \right).
    \]
    Using the first part of the proof, we get $\Pi_- \m(t,x) = -\frac{R_j^{\pm}}{(G)_{j,j}} + O\left( \frac{1}{t} \right)$. Using the rationnal expression of $\Pi_- M$, we also have as $t \rightarrow \pm \infty$
    \[
    \Pi_- \m(t,x) = \frac{R_j^{\pm}}{-a_j^+} + O \left( \frac{1}{t} \right),
    \]
    which allows us to conclude  
    \[
    \boxed{
    a_j^+ = a_j^- = (G)_{j,j}.
    }
    \]
    We now deal with the case $(G)_{j,j}=0$. In this case, considering again $x=v_j t$, we have 
    \[
    \left\{
    \begin{aligned}
        &\mathcal{K}_j = \sum_{i\neq k} (G)_{i,j} [G+tL-x]^{-1} \mathcal{K}_i,\\
        &g_k(t) + \sum_{i=1}^N \frac{(G)_{i,k}}{t(v_k-v_i)} g_i(t) = \frac{\mathcal{K}_k}{t(v_k-v_j)}.
    \end{aligned}
    \right.
    \]
    We hence have for $i\neq j$, $g_i(t) = [G+tL-x]^{-1} \mathcal{K}_i = O \left( \frac{1}{t} \right)+O \left( \frac{g_j(t)}{t} \right)$, which gives 
    
\end{proof}

\section{Absence of scattering to traveling waves}\label{sec:matTrav}

In this section, we show that if $\m$ is a solution of \eqref{HWMS}, then we can't have that $\m$ scatters to $g$ where $g$ is a traveling wave, except if $\m$ was actually traveling for all finite times. To do so, we first provide an algebraic characterization of being traveling. Then, we show that this property is transferred backward from $t=+\infty$ to finite times.

\subsection{Diagonal characterization of travelingness}

In this section, we provide characterizations of traveling wave solutions. Our focus is on algebraic properties that are preserved over infinite backward time. We leverage the Lax-Pair structure:
\begin{equation}
    \dot L(t) = [B(t),L(t)].
\end{equation}
Since on the diagonal, $L(t)_{i,i} = \dot x_i (t)$, then the poles are traveling if and only if $L$ has a constant diagonal with identical values. Since $\dot U(t) = B(t) U(t)$, the matrix $B(t)$ represents an infinitesimal change of coordinates at any given time.

For a wave to be considered traveling, two conditions must hold. First, all poles must move at a unique constant velocity $v$. This condition alone is not preserved, by the infinitesimal change of coordinates. Since the eigenvalues of $L$ are the asymptotic speeds, a necessary condition is that $v$ is an eigenvalue of $L(t)$ with multiplicity $N$. If $L(t)$ is a diagonal matrix $v I_N$, it is preserved under time evolution, as $L(t)$ is expressed through $L(t_0)$ and a time-dependent change of coordinates. However, this preservation over infinite time is nontrivial.

A traveling wave must also exhibit constant spins. While the condition $\dot S(t) =0$ While the condition  is necessary, it is not sufficient to ensure that the spins remain unchanged. Assuming the poles travel, the evolution equation:

\[
\dot s_j(t) = 2 \iu \sum_{k\neq j} \frac{s_j(t)\times s_k(t)}{(x_j(t)-x_k(t))^2},
\]
demonstrates the existence of degrees of freedom in spin interactions, allowing for soliton solutions with non-constant spins on top of traveling poles. The necessary condition $\dot S(t) = [B(t),S(t)]=0$ implies that the infinitesimal change of coordinates must commute with $S(t)$, which is not directly given by the poles traveling.

To formalize these ideas, we introduce two distinct notions of traveling waves.

\begin{definition}
    Let $\m(t,x)$ be a rational solution of \eqref{HWMS} of the form:

    \begin{equation*}
        \m(t,x) = \m_0 + i \sum_{j=1}^N \frac{s_j(t)}{x-x_j(t)} - i \sum_{j=1}^N \frac{\bar s_j(t)}{x-\bar x_j(t)},
    \end{equation*}
    and define $L(t)$, $S(t)$, $B(t)$ and $X(t)$ correspondingly as in \cite{matsuno2022integrability}. 

    We say that the \textbf{poles are traveling} $(PT)$ at velocity $v$ if and only if for any $t \in \mathbb{R}$,
    \begin{equation}\tag{PT}
        X(t) = X(0) + t v I_N.
    \end{equation}
    This condition ensures that $\ddot X(t) =0$ and $L(t)_{i,i}=v$.

    We now say that the wave is \textbf{completely traveling} $(CT)$ at speed $v$ if and only if 
    \begin{equation}\tag{CT}
        X(t) = X(0) + t v I_N,~\dot L(t) = \dot B(t) = \dot S(t) =0.
    \end{equation}
    
\end{definition}
The second notion is stronger and aligns with the standard notion for instance found in \cite{gerard2018lax} where a classification is provided. Assuming that $L$, $B$ and $X$ completely describe $\m$, then the second characterization corresponds to the more usual definition of traveling waves
\begin{equation*}
    \m(t,x) = Q_v (x-tv).
\end{equation*}
Our goal is to demonstrate that a completely traveling wave necessarily maintains a diagonal matrix $L(t)$ for all time.

Finally, since $\dot X(t) = L(t) + [B(t),X(t)]$,
$\dot X(t)$ is a diagonal matrix, and $[B(t),X(t)]$ has vanishing coefficients on the diagonal, $L$ will be a diagonal if and only if $[B(t),X(t)]=0$. $[B(t),X(t)]$ is hence a measure of how not diagonal $L(t)$ is. Note that $B(t)$ and $X(t)$ will never commute, except if $B(t)=0$. 

Differentiating the commutator:

\begin{equation}\label{evocrochet}
    \partial_t [B(t),X(t)] = [\dot B(t), X(t)] + [B(t),L(t)] + [B(t),[B(t),X(t)]],
\end{equation}
or equivalently 
\begin{equation}
    \partial_t [B(t),X(t)] = [\dot B(t),X(t)] + [B(t), \dot X(t)],
\end{equation}
If $\dot X(t)=t vI_N$, $\dot B(t)$ must vanish for $L(t)$ to stay diagonal. An equivalent way to write \eqref{evocrochet} is 
\begin{equation}
    \partial_t [B(t),X(t)] = [\dot B(t),X(t)] + \dot L(t) + [B(t),[B(t),X(t)]].
\end{equation}
Assuming a completely traveling wave $(CT)$, then $[\dot B(t),X(t)] = 0$ and $\dot L(t)=0$ so
\begin{equation}
    [B(t),[B(t),X(t)]] = 0.
\end{equation}
It does not immediately imply $[B(t),X(t)] = 0$, we will show how we are still able to conclude. 

We now state the main theorem of this section, that provides that all the characterization are equivalent, and all the traveling definitions are equivalent. In particular, it states that there can be no solution with traveling poles but evolving spins.

\begin{proposition}\label{diag}
    Let $\m(t,x)$ be a rational solution of \eqref{HWMS}. 

    If $\m(t,x)$ is in a $(CT)$ situation, meaning
    \begin{equation}
        X(t) = X(0)+tv I_N,~ \dot L(t) = \dot B(t) =\dot S(t)= 0,
    \end{equation}
    then for all $t \in \mathbb{R}$, 
    \begin{equation}
        L(t) = v I_N.
    \end{equation}
    Additionally, if $\m(t,x)$ is in a $(PT)$ situation, meaning
    \begin{equation}
        X(t) = X(0)+ t v I_N,
    \end{equation}
    then $\m$ is in a $(CT)$ situation. Conversly, if $L(t_0)$ is diagonal for some $t_0$, then the poles are traveling. Hence,
    \begin{equation}
        (PT) \Rightarrow (CT) \Rightarrow ~(L=v I_N) \Rightarrow (PT).
    \end{equation}
    so $(CT)$, $(PT)$ and $(L=v I_N)$ are equivalent.
    
\end{proposition} 

\subsection{Proof of Proposition \ref{diag}}

In this subsection, we prove Proposition \ref{diag}. We start with $(L=v I_N) \Rightarrow (CT)$ (Lemma \ref{LCT}), we then proceed with $(PT)\Rightarrow (CT)$ (Proposition \ref{PTCT}), which is the hardest part, and conclude by showing that $(CT)\Rightarrow (L=v I_N)$ (Lemma \ref{PTL}).

\begin{lemma}\label{LCT}
    If there exists some $t_0 \in \mathbb{R}$ such that $L(t_0)= vI_N$, then the wave is completely traveling $(CT)$, meaning
    \begin{equation}
        \dot L(t) = \dot B(t) = 0,~ \dot X(t) = vI_N.
    \end{equation}
    Conversely, if the wave is traveling at constant speed $v$, then for all $t\in \mathbb{R}$
    \begin{equation}
        \chi(L(t))(\lambda) = Det ( \lambda I_N - L(t)) = (\lambda - v)^N.
    \end{equation}
    If we further assume that $L(t_0)$ is diagonalizable, then $L(t) = vI_N$ for all $t$. 
    
\end{lemma}

\begin{proof}
    We first assume that $L(0)=v I_N$. By the Lax equation, we have for all $t \in \mathbb{R}$

    \begin{equation}
        L(t) = U(t) L(0) U(t)^{-1} = vI_N.
    \end{equation}
    Next, considering the evolution of $X(t)$, we obtain 

    \begin{equation}
        X(t) = U(t) (t L(0) + X(0) ) U(t)^{-1} = tv I_N + U(t) X(0) U(t)^{-1},
    \end{equation}
    so 
    \begin{equation}
        Det (X(t) - tv I_N - \lambda I_N) = Det (X(0) - vt I_N).
    \end{equation}
    In addition, travelingness of the poles give by definition $X(t) = X(0) + t v I_N$. Differentiating the commutator $[B(t),X(t)]$ using the expression for $\dot X(t)$ and by direct differentiation of the commutator provides the two expressions for $\partial_t [B(t),X(t)]$:

    \begin{equation}
    \left\{
    \begin{aligned}
        &\partial_t [B(t),X(t)] = \ddot X(t) - \dot L(t) = 0,\\
        &\partial_t [B(t),X(t)] = [\dot B(t),X(t)] + [B(t),\dot X(t)] = [\dot B(t),X(t)],
    \end{aligned}
    \right.
    \end{equation}
    which provides $[\dot B(t),X(t)]= 0$ and, using the simple poles assumption, we obtain $ \dot B(t)=0$.

    We now look at the converse, and assume that the wave is traveling, i.e.
    \begin{equation}
        \dot X(t) = v I_N,~\dot L(t) = 0,~ \dot B(t) =0.
    \end{equation}
    We consider the expression
        $X(t) = U(t) (X(0) + t L(t_0)) U(t)^{-1}$ and deduce $
        \frac{1}{t} X(0) + vI_N  = U(t) \left(\frac{1}{t} X(0) + L(0) \right) U(t)^{-1}$.
    In particular,
    \begin{equation}
        Det\left( \frac{1}{t} X(0) + vI_N - \lambda I_N \right) = Det \left(\frac{1}{t} X(0) + L(0) - \lambda I_N\right).
    \end{equation}
    By continuity, the left-hand side converges to $(v-\lambda)^N$ while the right-hand side converges to $\chi (L(0))(\lambda)$. Finally, if $L(0)$ is diagonalizable, then $L(0) = P (v I_N) P^{-1} = v I_N$.

\end{proof}

\begin{proposition}\label{PTCT}
    If the poles are traveling $(PT)$, meaning that for all $t\in \mathbb{R}$,

    \begin{equation}
        X(t) = X(0) + tv I_N,
    \end{equation}
    then the wave is completely traveling $(CT)$, meaning
    \begin{equation}
        \dot L(t) = \dot B(t) = [B(t),X(t)] =0 
    \end{equation}
    for all times.
    
\end{proposition}

\begin{proof}

    Starting with $X(t) = X(0) + t v I_N$  again in combination with the expression 
    $$X(t) = U(t) (X(0) + t L(0) ) U(t)^{-1},$$
    instantiating $t=0$ yields  
    \begin{equation}
        X(0) = U(t) ( X(0) - t [B(0),X(0)] ) U(t)^{-1}.
    \end{equation}
    Now, $\ddot X(t) = 0$ gives
        $[B(t),L(t)] + \partial_t [B(t),X(t)] =0$. 
    Since $[B(t),\dot X(t)] =0$, we conclude

    \begin{equation}
        [B(t),L(t)] + [\dot B(t),X(t)]=0.
    \end{equation}
    Also, applying the commutator with $L(t)$ in $\dot X(t) = L(t) + [B(t),X(t)]$ gives
    \begin{equation}
        [L(t),[B(t),X(t)]]=0.
    \end{equation}
    Since $L(t) = U(t) L(0) U(t)^{-1}$, then

    \begin{equation}
        [B(t),X(t)] = U(t) [B(0),X(0)] U(t)^{-1}.
    \end{equation}
    Now, differentiating 
    \begin{equation}
        U(t)^{-1} X(0) U(t) - X(0) = -t [B(0),X(0)].
    \end{equation}
    gives

    \begin{equation}
        U(t)^{-1} [B(t),X(0)] U(t) =  [B(0),X(0)],
    \end{equation}
    differentiating a second time and applying $t=0$ gives

    \begin{equation}\label{commuderiv}
        [B(0),[B(0),X(0)]] - [\dot B(0),X(0)] = 0.
    \end{equation}
    Since 
    \begin{equation}
        X(t) = U(t) (X(0) - t [B(0),X(0)] + t v I_N ) U(t)^{-1},
    \end{equation}
    we obtain using $X(t) = X(0) + v t I_N$ the identities
    \begin{equation}
        X(0) + t [B(t),X(t)] = U(t) X(0) U(t)^{-1},~ X(0) = U(t) (X(0) - t [B(0),X(0)] ) U(t)^{-1}.
    \end{equation}
    We define the functional $H_2$ and consider its derivative 
    \begin{equation}
        H_2(t) = U(t)^{-1} X(0) U(t) - X(0),~ \partial_t H_2(t) = -[B(t),U(t)^{-1} X(0) U(t)].
    \end{equation}
    But we also have $\partial_t H_2(t) = - [B(0),X(0)]$, so the two expressions have to coincide and we get

    \begin{equation}\label{commu}
        [B(0),X(0)] = [B(t),U(t)^{-1} X(0) U(t)].
    \end{equation}
    Using the commutativity of $v I_N$ we compute
    \begin{equation}
        [B(t),X(t)] = U(t) [B(0),X(0)] U(t)^{-1} = [U(t) B(t) U(t)^{-1}, X(0)].
    \end{equation}
    Since $[B(t),\dot X(z)] = 0$, we get 
        $[B(t) - U(t) B(t) U(t)^{-1},X(0)] = 0$.
    We now obtain since $X$ is a diagonal with distinct elements, 

    \begin{equation}
        B(t) = U(t) B(t) U(t)^{-1},~ [B(t),U(t)] =0.
    \end{equation}
    From this,

    \begin{equation}
        [B(t),X(t)] = [B(0), U(t) X(0) U(t)^{-1}].
    \end{equation}
    Using \eqref{commuderiv} and differentiating 

    \begin{equation}
        U(t) [B(t),X(t)] U(t)^{-1} = [B(0), X(0)],
    \end{equation}
    with respect to $t$ yields

    \begin{multline}
        [\dot B(t), U(t)^{-1} X(0) U(t) ] = [B(t),[B(t),U(t)^{-1} X(0) U(t)]] \\
        = U(t)^{-1} [\dot B(t), X(0)] U(t) = [U(t)^{-1} \dot B(t) U(t), U(t)^{-1} X(0) U(t)],
    \end{multline}

    and so

    \begin{equation}
         [U(t) \dot B(t) U(t)^{-1}, X(0)] = [\dot B(t),X(0)],
    \end{equation}
    which means that $\dot B(t)$ also commutes with $U(t)$. Since
    \begin{equation}
        U(t)^{-1}[B(t),X(t)]U(t) - [B(t),X(t)] = t [B(t),[B(t),X(t)]],
    \end{equation}
    so 
    
    we see that the derivative behaves as constant $[B(t),[B(t),X(t)]]$ instead of\\
    $[B(z),[B(z),X(z)]]$. Differentiating $-[B(t),X(t)]$ gives
    \begin{equation}
        \partial_t [B(t),[B(t),X(t)]]]=0,
    \end{equation}
    so $\ddot B(t) = 0$. Since 
    \begin{equation}
        [B(t),[B(t),X(t)]] = - [B(t),L(t)],
    \end{equation}
    $\ddot L(t)=0$. We obtain using the Lax relation for $L$ that $[\dot B,\dot L]=0$ and $[\dot B,[\dot B,\dot X]] = 0$.

    Differentiating the relation for $X(0)$ gives

    \begin{equation}
        [B(t),U(t) X(0) U(t)^{-1}] = [B(t),X(0)] + t [\dot B,X(0)],
    \end{equation}
    and differentiating again yields 
    \begin{equation}
        [B(t),[B(t),U(t)X(0)U(t)^{-1}]] - [B(t),[B(t),X(t)]] = [\dot B,X(0)] + t [\ddot B(t),X].
    \end{equation}
    Applying $t=0$ now gives since $U(0)=I_N$, 

    \begin{equation}
        [\dot B,X(0)] = 0 \Rightarrow \dot B=0.
    \end{equation}
    We now have that $B$ is constant. We obtain that $L$ is constant since 

    \begin{equation}
        L(t) = v I_N - [B(t),X(t)],
    \end{equation}
    and $x_i(t) - x_j(t)$ is constant. Since $L$, and $B$ are constant, we know that everything is traveling. We deduce that $L(t)$ is diagonal and 
    \begin{equation}
        [B(t),X(t)] = B(t) = S(t) = 0.
    \end{equation}

\end{proof}

\begin{theorem}\label{PTL}
    If $\m$ is completely traveling $(CT)$, then for any $t\in \mathbb{R}$,

    \begin{equation}
        L(t) = vI_N.
    \end{equation}
\end{theorem}

\begin{proof}
    Since $\m$ is completely traveling $(CT)$, so $B(t)=B$ is constant. Since $\dot L(t)=0$, we have:
    \begin{equation}
        [B,L(t)] = 0.
    \end{equation}
    Using $X(t) = X(0)+t v I_N$, we obtain $L=\dot X - [B,X(0)]$ and since $[B,\dot X]=0$, it follows that: 
    \begin{equation}
        [B,[B,X]] =0.
    \end{equation}
    The time evolution of $X$ is also given by 
    \begin{equation}
        X(t) = U(t) (X(0)+tv I_N - [B,X])U(t)^{-1},
    \end{equation}
    Hence,

    \begin{equation}
        X(0) = U(t) X(0) U(t)^{-1} - U(t) [B(t),X(0)] U(t)^{-1} = U(t) X(0) U(t)^{-1} - A(t).
    \end{equation}
    Our goal is now to show that $U(t)$ commutes with $X$. Since $B(t)=B$ is constant, $U(t) = e^{tB}$ commutes with $B$, leading to:
    \begin{equation}
        A(t) =  [B(t),U(t)X(0)U(t)^{-1}].
    \end{equation}
    Differentiating with respect to time yields

    \begin{multline}
        \partial_t A(t) = [B,BU(t)X(0)U(t)^{-1}- U(t)X(0)U(t)^{-1}B] \\
        = [B,[B,U(t)X(0)U(t)^{-1}]] = [B,A(t)],
    \end{multline}
    $A(t)$ hence satisfies a Lax equation and is given by, $A(t) = U(t) A(0) U(t)^{-1}$. Since $A(0) = [B,[B,X]] = 0$, we obtain $\Rightarrow A(t)=0$.
    Finally, this implies

    \begin{equation}
        X(0) = U(t) X(0) U(t)^{-1},
    \end{equation}
    so $[U(t),X]=0.$ Differentiating this expression gives and using $[B,U]=[X,U]=0$, we conclude
    \begin{equation}
        B U(t) X - X B U(t) = U(t) [B,X] = 0.
    \end{equation}
    Finally,

    \begin{equation}
        L(t) = \dot X(t) - [B,X] = v I_N.
    \end{equation}
    
\end{proof}

\subsection{Transfer from $+\infty$}

In this section, we bring together the results of the previous sections to provide a proof to Theorem \ref{thm:scatterTrav}. We will use that convergence in Sobolev norm as $t \rightarrow +\infty$ implies convergence of the coefficients, up to a subsequence and a permutation, see Appendix \ref{norm}. In addition, traveling waves are functions represented by a diagonal matrix for $L$ and a vanishing matrix for $B$. We recall Theorem \ref{thm:scatterTrav} and go on with the proof.

\trav*

\begin{proof}

    We first use Proposition \ref{diag} to deduce the shape of the different matrices corresponding to $G$. Indeed, since $G$ is traveling, we obtain
    \begin{equation}
        L_G(t) = v I_N,\quad B_G(t)=S_G(t) = 0.
    \end{equation}
    This implies $\textbf{c}_j \cdot \textbf{c}_k = 0$ for $j \neq k$. Corollary \ref{bij} states that the convergence in norm $||V(t,x) - G(t,x) || \rightarrow 0$ implies in particular the convergence in coefficients up to a bijection and a subsequence. Let then be $\phi$ and $t(n)$ such that 
    \begin{equation}
        A_j(t(n)) \rightarrow C_{\phi(j)},~ x_j(t(n)) - vt - y_{\phi(j)} \rightarrow 0.
    \end{equation}
    In particular, for $j\neq k$,
        $s_j(t(n)) \cdot s_k(t(n)) \rightarrow c_{\phi(j)} \cdot c_{\phi(k)} =0$.
    Also, we have 
    \begin{equation}
        \frac{1}{x_i(t) - x_j(t)} \rightarrow \frac{1}{y_{\phi(i)} - y_{\phi(j)}}.
    \end{equation}
    Hence, for $t$ large enough, 
        $|x_i(t)-x_j(t)|^{-1} \leq C$.
    Since $\partial_t \sum_{j=1}^N \im(x_i(t)) =0$ and $\im(x_i(t)) \geq 0$, then $\im(x_i(t)) \leq C$. Hence, 
        $||B_F(t)|| \leq C ||S_F(t)||$.
    Now, using $\partial_t S_F(t) = [B_F(t),S_F(t)]$, we get $\partial_t \left|  S_F(t) \right| \leq C |S_F(t)|^2$.
    This differential inequality is not compatible with $S_F(t) \rightarrow S_G = 0$, except if for all times, $S_F(t) = 0$. This implies in particular that
    \begin{equation}
        B_F(t)=0 \Rightarrow \partial_t L_F(t) = 0,
    \end{equation}
    and $L_F(t) = \dot X(t) - [B_F(t),X(t)] = \dot X(t)$, so the poles are traveling at a constant speed, and  writing
    \begin{equation}
        \chi(L_F(t)) \rightarrow \chi(L_G(t)) = \chi(v I_N)
    \end{equation}
    gives that all the speeds are equal to $v$.
    
\end{proof}

We now state another version of the theorem. While the first version showed that if a wave scatters to a traveling wave in Sobolev norm, then it had to be traveling for all times, this version shows that if the corresponding matrices converge to traveling ones, then the wave has to be traveling for all time. 

\begin{theorem}
    The solution of the (HWM) equation can't scatter in coefficients at $+\infty$ to a traveling wave without being a traveling at all time. More precisely, if
    \begin{equation}
        L_\infty = v I_N,~ B_\infty = 0,~ \dot X_\infty = v I_N, S_\infty=0,
    \end{equation}
    and 
    \begin{equation}
    \left\{
    \begin{aligned}
        &L(t) \xrightarrow[t\rightarrow \infty]{} L_\infty,\\
        &B(t) \xrightarrow[t \rightarrow \infty]{} B_\infty,\\
        &S(t) \xrightarrow[t \rightarrow \infty]{} S_\infty,\\
        &X(t)-vI_N \xrightarrow[t \rightarrow \infty]{} X_\infty,\\
    \end{aligned}
    \right.
    \end{equation}
    then at any time,
    
    \begin{equation}
        L(t) = v I_N,~ B(t) = 0,~ \dot X(t) = v I_N, S(t)=0.
    \end{equation}
    We also assume simple poles, i.e the diagonal matrix $X_\infty$ has distinct elements.
    
\end{theorem}

\begin{proof}
    For $t$ large enough, using $X(t) - v I_N \rightarrow X_\infty$, then with
    \begin{equation}
        \delta = \min_{i\neq j} |x_i(\infty) - x_j(\infty)|,
    \end{equation}
    we have 
        $|x_i(t) - x_j(t)| \geq \delta$. $B(t)$ can then be upper-bounded using $|S(t)|$ and an absolute constant, indeed

    \begin{equation}
        |B_{i,j} (t)| = \frac{|\sqrt{2s_i(t) \cdot s_j(t)}|}{|x_i(t)-x_j(t)|^2} \leq \frac{\sqrt{2 s_i(t) \cdot s_j(t)}}{\delta^2},
    \end{equation}
    which means that we also have for $B$
        $|B(t)| \leq C |S(t)|$.
    Now, since we have 
        $\dot S(t) = [B(t),S(t)]$,
    we obtain
        $\partial_t ||S(t)|| \leq C ||S(t)||^2$.
    In particular, we can't have
        $||S(t)|| \rightarrow 0$
    except if for all time $t\in \mathbb{R}$, 
        $S(t)=0$.
    This implies that $B(t)=0$, and $L(t) = \dot X(t)$ gives that $L(t)$ is a diagonal matrix. Since the eigenvalues are constant and are all equal to $v$, we obtain $L(t) = v I_N$, implying that the wave is traveling for all time.
    
\end{proof}

\printbibliography

\appendix

\section{Technical results}

\subsection{Calogero-Moser system for the $2\times 2$ formulation}

With

\begin{equation}
    U(t,x) = U_\infty + \sum_{j=1}^N \frac{A_j(t)}{x-\bar p_j(t)} + \sum_{j=1}^N \frac{A_j^*(t)}{x-p_j(t)},
\end{equation}
that

\begin{equation}
    |\nabla| U(x,t) = \sum_{j=1}^N \frac{i A_j(t)}{(x-\bar p_j(t))^2} - \sum_{j=1}^N \frac{i A_j^*(t)}{(x-p_j(t))^2},
\end{equation}
\begin{equation}
    \partial_t U(x,t) = \sum_{j=1}^N \left( \frac{\partial_t A_j(t)}{x-\bar p_j(t)} + \frac{\partial_t \bar p_j(t) A_j(t)}{(x-\bar p_j(t))^2} + \frac{\partial_t A_j^*(t)}{x-p_j(t)} + \frac{\partial_t p_j(t) A_j^*(t)}{(x-p_j(t))^2} \right).
\end{equation}
First, 

\begin{equation}
    [U_\infty,|\nabla| U] = i \sum_{j=1}^N \frac{[U_\infty,A_j(t)]}{(x-\bar p_j(t))^2} - i \sum_{j=1}^N \frac{[U_\infty,A_j^*(t)]}{(x-p_j(t))^2}.
\end{equation}
Now,

\begin{multline}
    [V,|\nabla| U] = \sum_{j=1}^N \sum_{k=1}^N \frac{1}{x-\bar p_j(t)} \frac{i}{(x-\bar p_k(t))^2} [A_j(t),A_k(t)] \\
    - \frac{1}{x-\bar p_j(t)} \frac{i}{(x-p_k(t))^2} [A_j(t),A_k^*(t)] + \frac{1}{x-p_j(t)} \frac{i}{(x-\bar p_k(t))^2} [A_j^*(t),A_k(t)] \\
    - \frac{1}{x-p_j(t)} \frac{i}{(x-p_k(t))^2} [A_j^*(t),A_k^*(t)].
\end{multline}

Also, 

\begin{equation}
    \frac{1}{(x-a)^2} \frac{1}{(x-b)} = \frac{1}{a-b} \frac{1}{(x-a)^2} + \frac{1}{(a-b)^2} \left( \frac{1}{x-b} - \frac{1}{x-a} \right).   
\end{equation}
We then obtain

\begin{equation}
    \partial_t A_j(t) = \left( \frac{-i}{2} \right) \sum_{k \neq j} \frac{2i [A_k(t),A_j(t)]}{(\bar p_k(t)-\bar p_j(t))^2} = \sum_{k \neq j} \frac{ [A_k(t),A_j(t)]}{(\bar p_k(t)-\bar p_j(t))^2}.
\end{equation}
and 

\begin{equation}
    \partial_t \bar p_j(t) A_j(t) = \left( \frac{-i}{2} \right) i [B_j(t),A_j(t)] = \frac{1}{2}[B_j(t),A_j(t)],
\end{equation}
where

\begin{equation}
    B_j(t)= U_\infty + \sum_{k\neq j} \frac{A_k(t)}{\bar p_j(t) - \bar p_k(t)} + \sum_{k=1}^N \frac{A_k^*(t)}{\bar p_j(t) - p_k(t)}.
\end{equation}

\section{Sobolev norms of rational functions}

\subsection{Proof of Lemma \ref{scalar}}\label{proof1}

	\begin{proof}
		We have that 
		\begin{equation}
			||\m(t)||_{L^2}^2 = \langle \m(t),\m(t) \rangle = \sum_{i,j=1}^N \frac{\s_i \bar \s_j}{x_i - \bar x_j} 2 i \pi \\
			= \sum_{i\neq j}^N \frac{\s_i \bar \s_j}{x_i - \bar x_j} 2 i \pi + \sum_{j=1}^N \frac{|\s_j|^2}{\im(x_j)} \pi.
		\end{equation}
		
		Now, the spins are bounded and $x_i - \bar x_j \rightarrow 0$ as $t \rightarrow \infty$, the result is obtained using standard calculus.

\subsection{Proof of Lemma \ref{technical}}\label{proof2}

	\end{proof}

    	\begin{proof}
		Formula \eqref{formulederiv} is obtained by a quick induction. We go on with the proof of \eqref{formule2}. It is obtained via changes of variables using the definition of the $L^2$ norm.
		
		\begin{multline}
			\langle \frac{\s_i}{(x-x_i)^{k+1}}, \frac{\s_i}{(x-x_i)^{k+1}} \rangle = \int_{-\infty}^\infty \frac{|s_i|^2}{|x-x_i|^{2k+2}} dx \\
			= \int_{-\infty}^\infty \frac{|s_i|^2}{((Re(x_i)-x)^2 + \im(x_i)^2)^{k+1}} dx = \int_{-\infty}^\infty \frac{|s_i|^2}{(z^2 + \im(x_i)^2)^{k+1}} dz \\
			= \frac{|s_i|^2}{\im(x_i)^{2k+1}} \int_{-\infty}^\infty \frac{1}{(z^2+1)^{k+1}} dz = \frac{|s_i|^2}{\im(x_i)^{2k+1}} \frac{\sqrt{\pi} \Gamma(k+\frac{1}{2})}{\Gamma(k+1)}.
		\end{multline}
		
		We now go on with the proof of \eqref{formule3}. For this, we use that $(x-x_i)$ and $(x-x_j)$ can not be both small at the same time, since $x_i-x_j$ is big. We assume without loss of generality that $Re(x_i) < Re(x_j)$. 
		
	On the interval $(-\infty,(Re(x_i)+Re(x_j))/2]$, we have that $|x-x_j| \geq |Re(x_i)-Re(x_j)|/2$. On the complementary, we have $|x-x_i| \geq |Re(x_i) - Re(x_j)|/2$. So we write
	
	\begin{multline}
		\langle \frac{\s_i}{(x-x_i)^{k+1}}, \frac{\s_j}{(x-x_j)^{k+1}} \rangle = \int_{-\infty}^\infty \frac{\s_i \bar \s_j}{(x-x_i)^{k+1} (x-\bar x_j)^{k+1}} dx \\
		= \int_{-\infty}^{\left(Re(x_i)+Re(x_j)\right)/2} \frac{\s_i \bar \s_j}{(x-x_i)^{k+1} (x-\bar x_j)^{k+1}}  dx \\
		+ \int_{\left(Re(x_i)+Re(x_j)\right)/2}^{\infty}\frac{\s_i \bar \s_j}{(x-x_i)^{k+1} (x-\bar x_j)^{k+1}}  dx,
	\end{multline}
		
	and 
	
	\begin{multline}
	\left|	\int_{-\infty}^{\left(Re(x_i)+Re(x_j)\right)/2} \frac{\s_i \bar \s_j}{(x-x_i)^{k+1} (x-\bar x_j)^{k+1}}  dx \right| \\ \leq \int_{-\infty}^\infty \frac{|s_i s_j |}{|x-x_i|^{k+1} \left((Re(x_j)-Re(x_j))/2 \right)^{k+1} } dx \\
	\leq C(k) \frac{|s_i s_j|}{\im(x_i)^k} \frac{1}{\left((Re(x_j)-Re(x_j))/2 \right)^{k+1} } \rightarrow 0.
	\end{multline}
		
	The other integral is similar.
		
	\end{proof}

\section{Convergence in Sobolev norm implies convergence in coefficients}\label{norm}

In this subsection, we establish that convergence in Sobolev norm to an "infinite-regime" implies convergence in coefficients. We consider two situations, one in which all the speeds are distinct, and a second one in which all the speeds are equal. In the second case, we must assume that the poles are different.

\begin{theorem}
    Let $\m(t,x)$ be a solution of the (HWM) equation, and be of the form
    \begin{equation}
        \m(t,x) = \m_0 + \sum_{j=1}^N \frac{A_j(t)}{x-x_j(t)} + \sum_{j=1}^N \frac{A_j^*(t)}{x - \bar x_j(t)},
    \end{equation}
    and we assume the existence of $G(t,x)$, an "infinite regime" function, of the form
    \begin{equation}
        G(t,x) = \m_0 + \sum_{j=1}^N \frac{C_j}{(x-v_j t) - y_j} + \sum_{j=1}^N \frac{C_j^*}{(x-v_j t) - \bar y_j},
    \end{equation}
    such that 
    \begin{equation}
        ||\m(t,x) - G(t,x) ||_{\dot H^{1/2}}^2 \xrightarrow[t \rightarrow \infty ]{} 0.
    \end{equation}
    and that the spins $A_j$ are bounded. Assuming either that all the speeds are different, or that all the speeds are equal
    \[
    (I):~\forall j\neq k,~v_j\neq v_k,\quad (II):~\forall j,k,~v_j=v_k=v,  
    \]
    then for every $j$, $A_j(t)$ has a limit and

    \begin{equation}
        A_j(t) \xrightarrow[t \rightarrow \infty ]{} C_j
    \end{equation}
    and the differences between the poles and their speed have a limit, and 
    \begin{equation}
        x_j(t)-v_j \xrightarrow[t \rightarrow \infty ]{} y_j.
    \end{equation}
    
\end{theorem}

\begin{proof} We focus on case $(I)$, as $(II)$ will be done later.

\textbf{Case $(I)$:}\\
    We use the time evolution equation
    \begin{equation}
        \dot s_j(t) = 2i \sum_{k\neq j}^N \frac{s_j(t) \times s_k(t) }{(x_j(t)-x_k(t))^2}.
    \end{equation}
    Since all the speeds are different by assumptions, $L(0) = P D P^{-1}$, where $D=diag(v_1,\dots,v_N)$.
    
    Since $X(t) = U(t) (X(0) + t L(0)) U(t)^{-1}$, 
    \begin{equation}
        X(t) = U(t) P^{-1} ( P X(0) P^{-1} + t D ) P U(t)^{-1}.
    \end{equation}
    The poles are the eigenvalues of $(PX(0)P^{-1} + t D)$, which is a diagonal dominant matrix, as the diagonal coefficients are $v_i t + (PX(0)P^{-1})_{i,i} $, and the other coefficients are given by $PX(0)P^{-1}_{i,j}$ and do not depend on $t$. Thus,
    \begin{equation}
        x_j(t) \in \mathcal{B} (tv_j,R),~ R=\max_j \sum_{k} \left| (PX(0)P^{-1})_{j,k} \right|.
    \end{equation}
    If for some $j$, $v_j=0$, applying the same reasoning to $X(t)+3t I_N \sim PX(0)P^{-1}+t(D+3I_N)$ gives $x_j(t)+3t = t(v_j+3) + O(1)$, in all cases 
    \[
    x_j(t) = t v_j + O(1).
    \]
    We thus obtain
    \begin{equation}
        |x_j(t) - x_k(t)| \geq |v_j - v_k| t - 2 R,
    \end{equation}
    so 
    \begin{equation}
        \left| \partial_t |s_j(t)| \right| \leq 2 \sum_{k\neq j}^N \frac{|s_j(t)| |s_k(t)|}{(x_j(t)-x_k(t))^2} \leq 2 ||s||_\infty \sum_{k\neq j} \frac{|s_j(t)|}{(|v_j-v_k|t  -2R)^2}.
    \end{equation}
    Since the right hand side belongs to $L_{1,t}$, then $s_j(t)$ converges to $s_j(\infty)$. Since the equation can be put into the form with $\eta = \min (|v_j-v_k|)$,
    \begin{equation}
        \left| \partial_t |s_j(t)| \right| \leq C \frac{|s_j(t)|}{(\eta t)^2},
    \end{equation}
    so in particular for $t_2 \geq t_1$,
    \begin{equation}
        \left| \int_{z=t_1}^{t_2} \frac{\partial_t |s_j(z)|}{|s_j(z)|} dz \right| \leq  \frac{C}{\eta^2} \left( \frac{1}{t_1} - \frac{1}{t_2} \right),~ \left| \ln\left( \frac{s_j(t_2)}{s_j(t_1)}\right) \right| \leq \frac{C}{t_1},
    \end{equation}
    so $s_j(\infty) \neq 0$.
    
    Now, we look at the time evolution equation for the poles. 
    \begin{equation}
        \ddot x_j(t) = -4 \sum_{k\neq j} \frac{s_j(t)\times s_k(t)}{(x_j(t)-x_k(t))^3},
    \end{equation}
    so 
    \begin{equation}
        \left| \ddot x_j(t) \right| \leq 4 \sum_{k\neq j} \frac{||s||_\infty^2}{(\eta t - 2 R)^3} \leq \frac{C}{t^3},
    \end{equation}
    so $\dot x_j(t) \rightarrow \dot x_j(\infty) = v_j$ since $x_j(t)\sim v_j t$. We then obtain
    \begin{equation}
        \left| \int_{z = t_1}^{t_2} \ddot x_j(z) dz \right| = |\dot x_j (t_2) - \dot x_j(t_1)| \leq C \left( \frac{1}{t_1^2} - \frac{1}{t_2^2} \right),
    \end{equation}
    taking $t_2 \rightarrow \infty$ gives 
    \begin{equation}
        |v_j - \dot x_j(t_1)| \leq \frac{C}{t_1^2}.
    \end{equation}
    Finally, the function $x_j(t) - v_j t$ has a derivative in $L^1(\mathbb{R})$, so it converges as $t \rightarrow \infty$, 
    \begin{equation}
        x_j(t) - v_j t \rightarrow x_j(\infty).
    \end{equation}
    Now, we consider the norm of the difference between the solution and the function it is scattering to, with $y_j(t) = y_j + t v_j$,
    \begin{multline}
        ||G(t,x)-\m(t,x)||_{\dot H^{1/2}}^2 = -4 \pi \sum_{j,k=1}^N \frac{\langle A_j(t),A_k(t) \rangle}{(x_j(t)-\bar x_k(t))^2} \\
        +4 \pi \sum_{j,k=1}^N \frac{\langle A_j(t),C_k \rangle  }{(x_j(t)-\bar y_k(t))^2} +4\pi \sum_{j,k=1}^N \frac{\langle C_j,A_k(t) \rangle  }{(y_j(t)-\bar x_k(t))^2}  -4 \pi \sum_{j,k=1}^N \frac{\langle C_j,C_k \rangle}{(y_j(t)-\bar y_k(t))^2}.
    \end{multline}
    For any $j \neq k$, the cross terms are $o(1/t^2)$ since the numerators are bounded and the speeds are different. Hence
    \begin{multline}
        ||G(t,x)-\m(t,x)||_{\dot H^{1/2}}^2 = o(1/t^2) + \pi \sum_{j=1}^N \frac{\langle A_j(t),A_j(t)\rangle}{Im(x_j(t))^2} +  \pi \sum_{j=1}^N \frac{\langle C_j,C_j \rangle}{Im(y_j)^2}  \\
        - \pi \sum_{j=1}^N \frac{\langle C_j,A_j(t) \rangle }{(y_j(t) -\bar  x_j(t))^2} -\pi \sum_{j=1}^N \frac{\langle A_j,C_j(t) \rangle }{(x_j(t) - \bar y_j(t))^2}.
    \end{multline}
    As $t \rightarrow \infty$,
    \begin{multline}
        ||G(t,x)-\m(t,x)||_{\dot H^{1/2}}^2 \rightarrow   \pi \sum_{j=1}^N \frac{\langle A_j(\infty),A_j(\infty)\rangle}{Im(x_j(\infty))^2} +  \pi \sum_{j=1}^N \frac{\langle C_j,C_j \rangle}{Im(y_j)^2}  \\
        - \pi \sum_{j=1}^N \frac{\langle C_j,A_j(\infty) \rangle }{(y_j -\bar  x_j(\infty))^2} -\pi \sum_{j=1}^N \frac{\langle A_j(\infty),C_j \rangle }{(x_j(\infty) - \bar y_j)^2}=0.
    \end{multline}
    Since we have by convexity for $a,b >0$, (with equality for $a=b$)
    \begin{equation}
        \frac{1}{(a+b)^2} \leq \frac{1}{8} \left( \frac{1}{a^2}+\frac{1}{b^2} \right),
    \end{equation}
    and using the parallelogram identity with $|x_j-\bar y_j| = |\bar x_j - y_j|$, 
    \begin{equation}
        2 |x_j - \bar y_j|^2 = 8 (Im(x_j)+Im(y_j))^2 + 8 (Re(x_j)-Re(y_j))^2,
    \end{equation}
    so 
    \begin{equation}
        \frac{1}{|x_j-\bar y_j|^2} \leq \frac{1}{2} \left( \frac{1}{Im(x_j)^2} + \frac{1}{Im(y_j)^2} \right).
    \end{equation}
    with equality for $Re(x_j) = Re(y_j)$. We deduce That for any $j$, $Re(x_j) = Re(y_j)$ and $Im(x_j)=Im(x_j)$, and that 
    \begin{multline}
        \frac{\langle A_j(\infty),A_j(\infty) \rangle}{ Im(x_j(\infty))^2 } +\frac{\langle C_j,C_j \rangle}{ Im(y_j(\infty))^2 } - \frac{\langle C_j,A_j(\infty) \rangle}{(y_j - \bar x_j(\infty))^2} - \frac{\langle A_j(\infty),C_j\rangle }{(x_j(\infty)-\bar y_j)^2} \\
        = \frac{1}{Im(x_j(\infty))^2}  \left( \langle A_j(\infty),A_j(\infty) \rangle + \langle C_j,C_j \rangle - 2 Re \langle A_j(\infty),C_j \rangle  \right) \geq 0,
    \end{multline}
    so for the norm to be zero, the terms must all be zero and 
    \begin{equation}
        (\langle A_j(\infty),A_j(\infty) \rangle^{1/2} - \langle C_j,C_j \rangle^{1/2}  )=0,
    \end{equation}
    so $||A_j(\infty)|| = ||C_j||$, and equality case in $\langle A_j(\infty_,C_j \rangle$ gives $A_j(\infty) = C_j$.
    
\end{proof}

Before doing the case $v_i = v_j$ for $i = j$, we need the following lemma.

\begin{lemma}\label{isolated}
    Let $F(t,x)$ be a rational function of the form
    \begin{equation}
        F(t,x) = \sum_{j,k=1}^N \frac{\langle A_j, A_k \rangle}{(x_j(t)-\bar x_k(t))^2},
    \end{equation}
    such that 
    \begin{equation}
        \langle F(t,x),|\nabla| F(t,x) \rangle \xrightarrow[t \rightarrow \infty]{} 0.
    \end{equation}
    Then, if $i_0$ is such that $x_{i_0}$ is isolated (or up to a subsequence), i.e.
    \begin{equation}
        |x_{i_0}(t_p) - x_j(t_p)|\geq \delta >0,
    \end{equation}
    for some $\delta >0$,
    then 
    \begin{equation}
        \langle A_{i_0}(t_{\phi(p)}),A_{i_0}(t_{\phi(p)}) \rangle \xrightarrow[p\rightarrow \infty]{} 0.
    \end{equation}
    
\end{lemma}

\begin{proof}

    We start by assuming the existence of $t_p \rightarrow \infty$ such that $|x_{i_0}(t_p)-x_j(t_p)|\geq \delta$ for any $p$ and $j\neq i_0$. We assume $i_0=1$. We have 

    \begin{multline}
        \langle F(t,x), |\nabla| F(t,x) \rangle = \sum_{j,k=1}^N \frac{\langle A_j(t),A_k(t) \rangle}{(x_j(t)-\bar x_k(t))^2} \\
        = \langle F(t,x-Re(x_1(t))),|\nabla| F(t,x-Re(x_1(t))) \rangle.
    \end{multline}
    
    We will distinguish two cases for the other poles $x_j$, $j \neq 1$. Either there exists a subsequence, such that $|x_j-x_1| \rightarrow \infty$, or $|x_j-x_1|$ is bounded. In the second case, we deduce that $x_j-x_1$ converges up to a subsequence. Call $J$ the first set of poles and $K$ the second one, and we consider a diagonal subsequence $\phi$ such that 
    \begin{equation}
    \left\{
        \begin{aligned}
            &|z_j|=|x_j-Re(x_1)| \rightarrow \infty,~ j \in J,\\
            &z_j=x_j-Re(x_1) \rightarrow z_j^\infty,~ j \in K. 
        \end{aligned}
    \right.
    \end{equation}
    Since we have by Plancherel's identity that $\langle F, |\nabla| F \rangle = \langle |\xi|^{1/2} \hat F, |\xi|^{1/2} \hat F \rangle = \langle |\nabla|^{1/2} F,|\nabla|^{1/2} F \rangle $, we look at $|\nabla|^{1/2} F$. With $S_j = \frac{A_j}{x-x_j}$, 
    \begin{equation}
        \mathcal{F}(S_j)(\xi) = \left\{
            \begin{aligned}
                &0,~ \xi >0,\\
                &A_j \cdot 2 i \pi e^{-2 i \pi \xi x_j},~ \xi\leq 0.
            \end{aligned}
        \right.
    \end{equation}
    Hence, 
    \begin{equation}
        |\nabla|^{1/2} \frac{A_i}{x-x_i} = c_{\nabla}\int_{\xi < 0 } e^{2 i \pi x \xi} \left( |\xi|^{1/2} A_i 2 i \pi e^{-2 i \pi \xi x_i} \right) dx = \frac{c_\nabla A_i}{4 \sqrt{2} \pi i^{3/2} (x-x_i)^{3/2} }.
    \end{equation}
    We prefer this formulation, as it is then possible to localize the contributions of the different solitons near their poles. We have for the Lebesgue norm 
    \begin{equation}
        || |\nabla| F(t,x) ||^2_{L^2} = \int_{x= -\infty}^\infty \left| \sum_{j=1}^N \frac{d_\nabla A_j}{(x-x_j)^{3/2}} \right| dx
    \end{equation}
    We now center at $x=Re(x_1)$ and distinguish the two cases for the poles. We also cut the integral to $[-1,1]$. 

    \begin{equation}
        || |\nabla|^{1/2} F(t,x) ||_{L^2}^2 \geq d_\nabla \int_{-1}^1 \left| \sum_{j\in J} \frac{A_j}{(x-x_j)^{3/2}}  + \sum_{j\in K}  \frac{A_j}{(x-x_j)^{3/2}} \right|^2 dx,
    \end{equation}
    where $x_1 \in K$ and $z_1 = x_1 - Re(x_1)$ by convention. With $\phi(p)$ such that $|z_j| \geq p $ for $j \in J$, then for $x \in [-1,1]$ and $j \in J$

    \begin{equation}
        \left| \frac{A_j }{(z_j-x)^{3/2}} \right| \leq  \frac{C}{p^{3/2}}.
    \end{equation}
    In particular, $F=F_1+F_2$ where $F_2$ converges uniformly to $0$ on $[-1,1]$, so $F_1$ must converge to $0$ in $L^2$ norm as well. Hence,

    \begin{equation}
          \int_{-1}^1 \left|  \sum_{j\in K} \frac{A_j }{(z_j-x)^{3/2}} \right|^2 dx \rightarrow 0.
    \end{equation}
    Now, $z_j \rightarrow z_j^\infty$ and $A_j \rightarrow A_j^\infty$, so we can apply dominated convergence for $p$ big enough and obtain

    \begin{equation}
        \int_{-1}^1 \left|  \sum_{j\in K} \frac{A_j }{(z_j-x)^{3/2}} \right|^2 dx \rightarrow \int_{-1}^1 \left|  \sum_{j\in K} \frac{A_j^\infty }{(z_j^\infty-x)^{3/2}} \right|^2 dx =0.
    \end{equation}
    In particular, 
    \begin{equation}
        \sum_{j\in K} \frac{A_j^\infty }{(z_j^\infty-x)^{3/2}} =0 
    \end{equation}
    almost everywhere on $[-1,1]$, and is continuous so everywhere on $[-1,1]$. Differentiating $k$ times at $x=0$ gives 
    \begin{equation}
        \sum_{j\in K} \frac{A_j^\infty}{(x_j)^{1/2}} \frac{1}{(x_j)^k} = 0.
    \end{equation}
    By regrouping the terms, $j \sim i$ when $x_i=x_j$, we know that $[1]=\{ 1\}$ and obtain the invertible Van-Der Monde determinant for $k \geq 1$
    \begin{equation}
        \sum_{[j]\in K/\sim} \left( \sum_{k \in [j]} \frac{A_k^\infty}{(x_k)^{3/2}} \right) \frac{1}{(x_{j})^{3/2}} = 0.
    \end{equation}
    This implies in particular that for any $j$
    \begin{equation}
        \sum_{k\in [j]} \frac{A_k^\infty}{(x_k)^{3/2}} =0,~ \Rightarrow A_1^\infty = 0.
    \end{equation}
    
\end{proof}

This theorem allows us to prove the more applicable corollary.

\begin{corollary}\label{bij}
    Let $F$ be a rational function of the form
    \begin{equation}
        F(t,x) = \sum_{j=1}^N \frac{A_j(t)}{x-x_j(t)},
    \end{equation}
    with bounded spins and such that $F$ scatters to some traveling function with distinct poles in Sobolev norm, i.e.
    \begin{equation}
        \left|\left| F(t,x) - \sum_{j=1}^N \frac{C_j}{x-v t - y_j} \right|\right|_{\dot H^{1/2}} \rightarrow 0,
    \end{equation}
    then, up to a subsequence, there exists a bijection $\phi: \{1,\dots,N\} \to \{1 ,\dots ,N\}$ such that each pole $x_j$ converges to the pole $vt+y_{\phi(j)}$, or more precisely
    \begin{equation}
        x_j(t) - v t - y_{\phi(j)} \rightarrow 0,
    \end{equation}
    and the corresponding spin also converges to the corresponding spin
    \begin{equation}
        A_j(t) \rightarrow C_{\phi(j)}.
    \end{equation}
\end{corollary}

\begin{proof}
    We will use the previous theorem, we start by defining the function 

    \begin{equation}
        G(t,x) = F(t,x) - \sum_{j=1}^N \frac{C_j}{x-vt - y_j},
    \end{equation}
    and by assumption 
    \begin{equation}
        ||G(t,x)||_{\dot H^{1/2}} \rightarrow 0.
    \end{equation}
    We remark that the corresponding spin to each $vt+y_j$ is $C_j$, which is constant and non-zero, so it does not converge to $0$. Hence, Lemma \ref{isolated} states that the pole can not be isolated, so there exists $\psi(j)$ (that will be $\phi^{-1}$) such that 
    \begin{equation}
        |x_{\psi(j)} - vt - y_j| \rightarrow 0.
    \end{equation}
    $\psi$ is indeed a bijection. Indeed, $\psi$ always chose poles $x_j$ and not $y_k$, because the distances between the $y_k$ are constant. Also, $\psi$ has to be surjective, as any pole $y_k$ can't be isolated at $+\infty$. Hence, by cardinality, $\psi$ is a bijection and we call $\phi$ its inverse.

    Now, using a similar argument as in the proof of Lemma \ref{isolated}, we regroup poles and obtain by dominated convergence the Vandermonde system up to a subsequence of times,
    \begin{equation}
        \sum_{j=1}^N \left( \frac{A_{\psi(j)}^\infty-C_j}{y_j} \right) y_j^k = 0,
    \end{equation}
    which gives $A_j^\infty = C_{\phi(j)} $ for all $j$.
    
\end{proof}

    % //////////////// NEW \\\\\\\\\\\\\\\\ 

    \begin{theorem}

        Let 
        \[
        \m(t,x) = \m_0 + \sum_{j=1}^N \frac{A_j(t)}{x-x_j(t)} + \sum_{j=1}^N \frac{A_j^*(t)}{x-\bar x_j(t)},
        \]
        such that $X(t)=U(t) \left[ X(0)+t L(0) \right]U^{-1}(t)$. Then, assuming $A_j$ bounded, we have
        \[
        \left|\left| \m(t,x)-\m_0 \right|\right|_{\dot H_x^{1/2}} \xrightarrow[t\rightarrow 0]{} 0 \quad \Rightarrow \quad A_j(t) \xrightarrow[t\rightarrow 0]{} 0.
        \]

    \end{theorem}

    \begin{proof}
        The strategy of the proof will be as follows. First, we write $L(0)$ in Jordan form. Then we will combine the arguments used in the case where all eigenvalues are distinct, and in the case when all the eigenvalues are the same. We call $\eta_i$ the multiplicity of $v_i$.
        
        We write for some $P \in GL_N(\mathbb{C})$,
        \[
        L(0) =P \mathcal{J}_\varepsilon P^{-1}= P \begin{bmatrix}
J_1 & \;     & \; \\
\;  & \ddots & \; \\ 
\;  & \;     & J_p\end{bmatrix} P^{-1},~ J_i = \begin{bmatrix}
v_i & \kappa_{i,1}            & \;     & \;  \\
\;        & v_i    & \ddots & \;  \\
\;        & \;           & \ddots & \kappa_{i,\eta_i-1}   \\
\;        & \;           & \;     & v_i       
\end{bmatrix}
        \]
    where for any $i,j$, $\kappa_{i,j} = 0 $ or $\varepsilon$. So, 
    \[
    X(t) = U(t) \left[ X(0) + t P \mathcal{J}_\varepsilon P^{-1} \right] U^{-1}(t)= U(t) P\left[ P^{-1} X(0) P + t \mathcal{J}_\varepsilon \right] P^{-1} U^{-1}(t).
    \]
    In particular, the eigenvalues of $X(t)+3tI_N$ are the eigenvalues of 
    \[
    P^{-1}X(0)P + t \mathcal{J}_\varepsilon + 3 t I_N = K_0 + t \mathcal{J}_\varepsilon + 3 t I_N .
    \]
    Since this matrix is diagonal dominant, we have for some $\phi$ giving the corresponding bloc,
    \[
    |x_j(t)-v_{\phi(j)} t + C_j| \leq |t \varepsilon + D_j|,
    \]
    for some constants $C_j$ and $D_j$, so $|x_j-v_{\phi(j)} t| \leq |t \varepsilon + D_j| + |C_j| \leq t \varepsilon + R $ for $R\geq 0$ bounded. We define $\sim$ the equivalence relation as $j\sim k$ if $\phi(j) \sim \phi(k)$. Then, for $[j]\neq [k]$,
    \[
    |x_i-x_j| \geq \left| |x_i-x_j - t v_{[i]} + t v_{[j]}| - |tv_{[i]}-t v_{[j]}|  \right|.
    \]
    Since $|x_i-x_j - t v_{[i]} + t v_{[j]}| \leq 2 t \varepsilon + 2R$, for $\varepsilon< \eta/2$ and $t$ large enough,
    \[
    |x_i-x_j| \geq t |v_{[i]}-v_{[j]}| - 2 t \varepsilon - 2 R \leq t (\eta - 2 \varepsilon) -2R.
    \]
    Hence, for any $i,j$ such that $[i]\neq [j]$ and $A$ and $B$ bounded,
    \begin{equation}\label{lollol}
    \frac{\langle A, B\rangle}{(x_j-\bar x_k)^2} \xrightarrow[]{t\rightarrow 0} 0.
    \end{equation}

    Now, we write 
    \begin{multline}
    f(t,x)+f^*(t,x)=\m(t,x)-\m_0 =  \sum_{j=1}^N \frac{A_j}{x-x_j} + \sum_{j=1}^N \frac{A_j^*}{x-\bar x_j} \\
    = \sum_{j=1}^p \sum_{k\sim j} \frac{A_k}{x- x_k} + \sum_{j=1}^p \sum_{k\sim j} \frac{A_k^*}{x-\bar x_k} = \sum_{j=1}^p \left( f_j + f_j^*\right).
    \end{multline}
    Hence, \eqref{lollol} gives in particular for $j\neq k$,
    \[
    \left\langle f_j,f_k \right\rangle_{\dot H^{1/2}} \xrightarrow[]{t\rightarrow 0} 0.
    \]
    Hence, $\langle f,f \rangle_{\dot H^{1/2}} \rightarrow 0$ gives in particular 

    \[
    \sum_{j=1}^p \langle f_j,f_j \rangle = o(1).
    \]
    Since all the involved terms are non-negative, we obtain for any $j$, $\langle f_j,f_j \rangle \rightarrow 0$. Using the aforementioned argument when all speeds are equal, regrouping using $i \simeq j$ if $x_i - x_j$ is bounded, we obtain $A_j \rightarrow 0$.

    \end{proof}

    \section{More characterizations for travelingness}

    We now define two properties $\mathcal{H}_1(t)$ and $\mathcal{H}_2(t)$. They are local properties and imply travelingness of the poles.

    \begin{definition}
We say that $\mathcal{H}_1(t)$ is satisfied at $t \in \mathbb{R}$ if 

\begin{equation}\tag{$\mathcal{H}_1(t)$}
    \forall \alpha \in \mathbb{R},~ X(t) \sim X(t) + \alpha[B(t),X(t)].
\end{equation}
We define also define $\mathcal{H}_2(t)$:
\begin{equation}\tag{$\mathcal{H}_2(t)$}
    \exists v\in \mathbb{R}, ~\dot X(t) = v I_N.
\end{equation}
\end{definition}

\begin{lemma}

    We consider $\m$ a solution of \eqref{HWMS}, and define $X$, $L$, $B$ correspondingly. Then, if there exists  $t_0$ such that $\mathcal{H}_1(t_0)$ and $\mathcal{H}_2(t_0)$ are satisfied, $(PT)$ is satisfied, so the poles are traveling.
    
\end{lemma}

\begin{proof}
    We assume without loss of generality that $t_0=0$. Using [refref], we write

\begin{equation}
    X(t) = U(t)^{-1} ( t L(0) + X(0)) U(t)= U(t)^{-1} ( t \dot X(0) + t [B(0),X(0)] + X(0)  )U(t)^{-1}.
\end{equation}
    Since $\dot X(0) = vI_N$ commutes with $U(t)$, we obtain

    \begin{equation}
        X(t) - v  t I_N = U(t) ( X(0) + t [B(0),X(0)] ) U(t)^{-1}.
    \end{equation}
    Since $X(t) - v t I_N$ is a diagonal matrix, the entries are equal to the eigenvalues of 
    \begin{equation}
        X(0) + t [B(0),X(0)],
    \end{equation}
    and $\mathcal{H}_1(0)$ gives

    \begin{equation}
        X(0) + t [B(0),X(0)] \sim X(0)
    \end{equation}
    for any $t$. In particular, the characteristic polynomials are identical and we obtain

    \begin{equation}
        X(t) - v I_N \sim X(0),
    \end{equation}
    so $X(t) = X(0)+ v I_N$.
    
\end{proof}

\end{document}